\documentclass[11pt,a4paper,notitlepage, oneside]{scrarticle}

\usepackage[utf8]{inputenc}
\usepackage{subfigure}
\usepackage{amsthm,amsmath,amssymb,amsthm}
\usepackage{mathtools}
\usepackage[ruled,linesnumbered, vlined]{algorithm2e}
\usepackage{authblk}
\usepackage{xargs}
\usepackage{xspace}
\usepackage{braket}
\usepackage{todonotes}

\usepackage{xcolor}
\colorlet{color2}{magenta}
\colorlet{color3}{blue}
\colorlet{color1}{cyan}
\colorlet{color4}{orange}
\colorlet{color5}{violet}

\usepackage{hyperref}
\hypersetup{
	colorlinks=true,  
	linkcolor=black,    
	urlcolor=black!70,     
	citecolor=black!70
}

\newlength{\nodedistance}
\usepackage{tikz}
\usetikzlibrary{shapes, positioning}

\newcommand{\labelLeft}[2]{\node[left of=#1, xshift=.5\nodedistance] (#1l) {#2};}
\newcommand{\labelRight}[2]{\node[right of=#1, xshift=-.5\nodedistance] (#1l) {#2};}
\newcommand{\labelAbove}[2]{\node[above of=#1, yshift=-.5\nodedistance] (#1l) {#2};}
\newcommand{\labelBelow}[2]{\node[below of=#1, yshift=.5\nodedistance] (#1l) {#2};}

\tikzstyle{vertex}=[line width=1pt, color=black, circle, fill=black, minimum size=5pt,inner sep=0pt,draw]
\tikzstyle{edge} = [-, thick, color=black!50]
\newcommand{\drawEdge}[2]{
	\draw[-, line width=1.75pt, color=white] (#1) edge (#2);
	\draw[edge] (#1) edge (#2);
}
\newcommand{\drawEdgeCol}[3]{
	\draw[-, line width=1.75pt, color=white] (#1) edge (#2);
	\draw[edge, #3] (#1) edge (#2);
}
\usetikzlibrary{intersections}
\usetikzlibrary{arrows.meta}
\usetikzlibrary{arrows, backgrounds, trees, calc}

\pgfdeclarelayer{background}
\pgfdeclarelayer{backbackground}
\pgfsetlayers{backbackground,background,main}

\newcommand{\convexpath}[2]{
	[   
	create hullnodes/.code={
		\global\edef\namelist{#1}
		\foreach [count=\counter] \nodename in \namelist {
			\global\edef\numberofnodes{\counter}
			\node at (\nodename) [draw=none,name=hullnode\counter] {};
		}
		\node at (hullnode\numberofnodes) [name=hullnode0,draw=none] {};
		\pgfmathtruncatemacro\lastnumber{\numberofnodes+1}
		\node at (hullnode1) [name=hullnode\lastnumber,draw=none] {};
	},
	create hullnodes
	]
	($(hullnode1)!#2!-90:(hullnode0)$)
	\foreach [
	evaluate=\currentnode as \previousnode using \currentnode-1,
	evaluate=\currentnode as \nextnode using \currentnode+1
	] \currentnode in {1,...,\numberofnodes} {
		let
		\p1 = ($(hullnode\currentnode)!#2!-90:(hullnode\previousnode)$),
		\p2 = ($(hullnode\currentnode)!#2!90:(hullnode\nextnode)$),
		\p3 = ($(\p1) - (hullnode\currentnode)$),
		\n1 = {atan2(\y3,\x3)},
		\p4 = ($(\p2) - (hullnode\currentnode)$),
		\n2 = {atan2(\y4,\x4)},
		\n{delta} = {-Mod(\n1-\n2,360)}
		in 
		{-- (\p1) arc[start angle=\n1, delta angle=\n{delta}, radius=#2] -- (\p2)}
	}
	-- cycle
}

\newtheorem{theorem}{Theorem}
\newtheorem{lemma}[theorem]{Lemma}
\newtheorem{corollary}[theorem]{Corollary}
\newtheorem{observation}[theorem]{Observation}
\newtheorem{proposition}[theorem]{Proposition}

\newtheorem*{conjecture*}{Conjecture}
\theoremstyle{definition}
\newtheorem{definition}[theorem]{Definition}

\usepackage{enumitem,xparse}
\newlist{Claim}{description}{2}
\setlist[Claim]{labelindent=2em,leftmargin=*}
\newcounter{claim}[theorem]

\let\originalqedsymbol\qedsymbol
\newenvironment{subclaim}{
	\renewcommand{\qedsymbol}{$\blacksquare$}
	\refstepcounter{claim}%
	\let\theclaimcounter\theclaim%
	\Claim\item[\textbf{Claim \theclaimcounter:}]%
}{\endClaim\let\qedsymbol\originalqedsymbol}

\newcommand{\tw}{\operatorname{tw}}
\newcommand{\td}{(\mathcal{B}, \mathcal{T})}

\newcommand{\card}[1]{\lvert#1\rvert}

\newcommand{\skein}[3]{$#1$-$#2$ $#3$-skein\xspace}
\newcommand{\skeins}[3]{$#1$-$#2$ $#3$-skeins\xspace}

\newcommand{\secondCoordP}{\textsc{Second Coordinate in Connectivity Pair}\xspace}
\newcommand{\scp}{\textsf{2-CP}\xspace}
\newcommand{\parvc}{\textsf{PVCB}\xspace}

\usepackage{enumitem}
\newenvironment{thmenum}
{\begin{enumerate}[label=(\roman*),ref=\thetheorem.(\roman*),noitemsep]}
	{\end{enumerate}}

\usepackage{graphicx}

\usepackage[numbers]{natbib}

\author{Sebastian S. Johann, Sven O. Krumke and Manuel 
Streicher\thanks{Corresponding Author, 
		streicher@mathematik.uni-kl.de}}
\affil[]{Technische Universit\"at Kaiserslautern}
\title{On the Mixed Connectivity Conjecture of Beineke and Harary}

\usepackage{tabularx,environ}
\usepackage{chngpage}
\makeatletter
\newcommand{\problemtitle}[1]{\gdef\@problemtitle{#1}}
\newcommand{\probleminput}[1]{\gdef\@probleminput{#1}}
\newcommand{\problemquestion}[1]{\gdef\@problemquestion{#1}}
\NewEnviron{decisionproblem}{
	\problemtitle{}\probleminput{}\problemquestion{}
	\BODY
	\par\addvspace{.5\baselineskip}
	\noindent\textbf{\@problemtitle}
	\begin{adjustwidth}{.8cm}{0cm}
		\hspace*{-.4cm}\textbf{Instance:} \@probleminput\\%
		\hspace*{-.4cm}\textbf{Question:} \@problemquestion%
	\end{adjustwidth}
	
	\par\addvspace{.5\baselineskip}
}
\makeatother

\begin{document}

\maketitle

\begin{abstract}
	\noindent The conjecture of Beineke and Harary states that for any two vertices which can be separated by $k$ vertices and $l$ edges for $l\geq 1$ but neither by~$k$ vertices and $l-1$ edges nor $k-1$ vertices and $l$ edges there are~$k+l$ edge-disjoint paths connecting these two vertices of which $k+1$ are internally disjoint.
	In this paper we consider this conjecture for~$l=2$ and any $k\in \mathbb{N}$.
	Afterwards, we utilize this result to prove that the conjecture holds for all graphs of treewidth at most $3$ and all $k$ and $l$.
	We also show that it is \textsf{NP}-complete to decide whether two vertices can be separated by~$k$ vertices and $l$ edges.
\end{abstract}
\section{Introduction}
\label{sec:introduction}

Connectivity is an extensively studied property of graphs. A well-known Theorem
of Menger establishes equality between the vertex connectivity for a given pair
of non-adjacent vertices and the maximum number of internally disjoint paths
between this pair as well as the edge connectivity for a given pair of vertices
and the maximum number of edge-disjoint paths between this pair. There are many
variation and extensions of Menger's Theorem. For example Aharoni and Berger
proved a version of Menger's Theorem for infinite graphs,
\textit{cf.}~\cite{AB08}, and Borndörfer and Karbstein interpreted and proved
Menger's Theorem in hypergraphs, \textit{cf.}~\cite{BK12}. In this paper we
focus on a form of connectivity in which vertices and edges may be removed at
the same time. One variant of \emph{mixed connectivity} was considered by Egawa,
Kaneko and Matsumoto~\cite{EKM91}. They prove the following mixed version of
Menger's Theorem: Between two vertices $v,w$ of a graph there are $\lambda$
edge-disjoint unions of $k$ internally disjoint paths if and only if for each
set $S$ of $0\leq r\leq\min\{k-1, |V(G)|-2\}$ vertices the graph $G-S$ contains
$\lambda(k-r)$ edge-disjoint $v$-$w$ paths. Beineke and Harary~\cite{B01}
proposed an alternative form of mixed connectivity between pairs of vertices.
They call a pair of non-negative integers $(k,l)$ \emph{connectivity pair} for
distinct vertices $s$ and $t$ if they can be separated by removing $k$ vertices
and $l$ edges, but neither by~$k$ vertices and $l-1$ edges nor $k-1$ vertices
and $l$ edges. In~\cite{B01} Beineke and Harary claim to have proved a mixed
version of Menger's Theorem: If $(k,l)$ is a connectivity pair for $s$ and $t$, then
there exist $k+l$ edge-disjoint $s$-$t$ paths $k$ of which are internally
disjoint. Mader pointed out in~\cite{Mader79}, that the proof is erroneous.
%
%

The most meaningful result on the conjecture by Beineke and Harary to date is
due to Enomoto and Kaneko, \textit{cf.}~\cite{enomoto1994}. They first extended
the conjecture claiming that it is possible to find $k+1$ internally disjoint
paths instead of just $k$ under the additional assumption that $l\geq 1$ and
then proved their statement for certain $k$ and $l$. The exact result is
restated as Theorem~\ref{thm:enomoto_kaneko} in Section~\ref{sec:preliminaries}.

From our studies the following conjecture originally formulated by Beineke and 
Harary in~\cite{B01} and extended by Enomoto and Kaneko in~\cite{enomoto1994} 
may hold. In the remainder of this article we refer to the conjecture by the 
name \emph{Beineke-Harary-Conjecture}.

\begin{conjecture*}[Beineke-Harary-Conjecture]
	\label{con:beineke_harary}
	Let $G$ be a graph, $s,t\in V(G)$ distinct vertices and $k,l$ non-negative 
	integers with $l\geq 1$. If $(k,l)$ is a connectivity pair for $s$ and $t$ 
	in $G$, then there exist $k+l$ edge-disjoint paths, of which $k+1$ are 
	internally disjoint.
\end{conjecture*}

Our main contribution is to prove the conjecture for $l=2$ and any
$k\in\mathbb{N}$. It is worth noting that for $l=2$ the conjecture has
\emph{not} been proved for any $k>1$. In particular, the result of Enomoto and
Kaneko does not apply to these cases and their proof does not appear to have an
easy adaption for these cases. The techniques used to prove the conjecture for
$l=2$ are novel. The main idea is to start with $k$ internally disjoint paths
and then find the missing path by inductively moving through the graph and
adjusting the $k$ internally disjoint paths whenever necessary. We observe that
our result, together with the result due to Enomoto and Kaneko,
\textit{cf.}~Theorem~\ref{thm:enomoto_kaneko}, implies that the
Beineke-Harary-Conjecture holds for $k=2$ and all $l\in\mathbb{N}$. We then
utilize this fact to prove the Beineke-Harary-Conjecture for all graphs that
have treewidth at most $3$. 

It is well known that the problem of determining the vertex- or
edge-connectivity between a pair of vertices is polynomial time solvable. It can
be done by computing maximum flows in related graphs. This leads to the obvious
question about the complexity of determining all connectivity pairs for a pair
of vertices. This question is also raised by Oellermann
in~\cite{beineke2012topics}. It can be observed that for a graph $G$, two
vertices $s$ and $t$ and sensibly chosen $k\in\mathbb{N}$, there is a unique
$l_k\in\mathbb{N}$ such that $(k,l_k)$ is a connectivity pair for $s$ and $t$ in
$G$, cf.~Observation~\ref{obs:second_coordinate_unique}. The approaches for
determining the vertex- or edge-connectivity between a pair of vertices do not
transfer to connectivity pairs: In this paper we give a formal proof that the
decision version of the raised question is \textsf{NP}-complete by reducing
\textsc{Bipartite Partial Vertex Cover} to it, which is known to be
\textsf{NP}-complete, \textit{cf.}~\cite{caskurlu}. This fact is surprising and
underlines the complex nature of mixed connectivity related problems. 

\paragraph*{Outline} After we state some basic definitions in
Section~\ref{sec:preliminaries}, we establish some preliminary results on connectivity
pairs and the Beineke-Harary-Conjecture in Section~\ref{sec:cp_foundations}. We
present the proof of our main result in the subsequent
Section~\ref{sec:main_result}. Section~\ref{sec:treewidth} focuses on proving
the conjecture on graphs of small treewidth. Finally, in
Section~\ref{sec:complexity} we prove the intractability of the computation of
the second coordinate in a connectivity pair.

\section{Preliminaries}
\label{sec:preliminaries}

Most of our notation is standard graph terminology as can be found
in~\cite{Diestel00, Wes01}. We recall some basic notations in the following. The
graphs under consideration may contain parallels but no loops. For a graph $G$,
we refer to the vertex set of the graph $G$ by $V(G)$ and to the edge set by
$E(G)$. We denote an edge joining vertices $u,v\in V(G)$ by $uv$. Note that the
exact choice of the edge, if parallel edges are present, is not of relevance to
any of our proofs. For $U,V\subseteq V(G)$ and $u\in U$ and $v\in V$ we call
$uv$ a $U$-$V$ edge. The set of all $U$-$V$ edges in $E(G)$ is denoted by
$E(U,V)$; instead of $E(\{u\},V)$ and $E(U,\{v\})$ we write $E(u,V)$ and
$E(U,v)$.  If $H$ is another graph we denote by $G\cup H$ the graph with vertex
set $V(G)\cup V(H)$ and edge set $E(G)\cup E(H)$, where we assume that equal
edges join the same set of endvertices. For a subset of vertices $S\subseteq
V(G)$ we denote by $G[S]$ the graph \emph{induced} by $S$, that has vertex set
$S$ and all edges joining vertices of $S$. Further, we denote by $G-S$ the graph
$G[V(G)\setminus S]$. For a subset $E^\prime\subseteq E$ we write $G-E^\prime$
for the graph with vertex set $V(G)$ and edge set $E\setminus E^\prime$. To
simplify notation we write $G-v$ and $G-e$ instead of $G-\{v\}$ and $G-\{e\}$
for $v\in V(G)$ and~$e\in E(G)$.

A \emph{path} $P=v_0\dots v_k$ is a graph with vertex set $\{v_0,\dots,v_k\}$
and edge set of the form $\set{v_iv_{i+1}\colon i=0,\dots,k-1}$, where all vertices are
distinct except possibly $v_0$ and~$v_k$. If $v_0\neq v_k$ we refer to $P$ as a
$v_0$-$v_k$ path. We denote by $v_iPv_j$ with $i \leq j$ the subpath
$v_iv_{i+1}\ldots v_j$. If $v_i=v_0$ ($v_j=v_k$) for simplicity of notation we
also write $Pv_j$ ($v_iP$). Two or more paths are \emph{edge-disjoint} if no two
paths use the same edge. Two or more $s$-$t$ paths are \emph{internally
	disjoint} if they only share the vertices $s$ and $t$. If $P_1,\dots,
P_k$ are internally disjoint $s$-$t$ paths, we call the graph
$\bigcup_{i=1}^k P_i$ an \emph{\skein{s}{t}{k}}.

For distinct vertices $s$ and $t$ we say that a set $W\subseteq
V(G)\setminus\{s,t\}$ ($F\subseteq E(G)$) \emph{separates} $s$ and~$t$ in $G$ if
$s$ and~$t$ are not connected in $G-W$ ($G-F$). In this case we call~$W$ ($F$) an
\emph{$s$-$t$ vertex-(edge-)separator}. If $s$ and $t$ are non-adjacent, we denote by
$\kappa_G(s,t)$ the size of a smallest vertex-separator for $s$ and~$t$, where
we omit the subscript $G$ if the graph is clear from context. For a graph $G$ a
set $W$ is a \emph{vertex-separator} if $G-W$ is not connected.

\section{Connectivity Pairs and Foundations of the Beineke-Harary-Conjecture}
\label{sec:cp_foundations}

\begin{samepage}
In this section we provide the formal definition for connectivity pairs, recall
some basic results on the Beineke-Harary-Conjecture, and establish further basic
results on mixed separators and the conjecture.
\end{samepage}

\newpage

\begin{samepage}
\begin{definition}[Disconnecting Pair]
	Let $G$ be a graph and $S,T\subseteq V(G)$. We call a pair $(W,F)$ with 
	$W\subseteq V(G)\setminus \left(S\cup T\right)$ and $F\subseteq E(G)$ an 
	$S$-$T$ \emph{disconnecting pair} if in~$G-W - F$ there 
	is no path from a vertex in $S$ to a vertex in $T$.
	
	We call the number of edges in a disconnecting pair its \emph{size}, the 
	number of vertices in a disconnecting pair its \emph{order} and the number 
	of elements $\card{W}+\card{F}$ to be its \emph{cardinality}.
	
	If $S=\{s\}$ or $T=\{t\}$ consists of only one element we omit the set 
	brackets in the notation and also write $s$-$t$ disconnecting pair.
\end{definition}
\end{samepage}

Beineke and Harary introduced \emph{connectivity pairs} in their paper from 
1967~\cite{B01}. We recall their definition in the following.

\begin{definition}[Connectivity Pairs]
	Let $G$ be a graph and $s,t\in V(G)$ distinct vertices. 
	We call an ordered pair of non-negative integers $(k,l)$ a \emph{connectivity pair} 
	for $s$ and $t$ in~$G$ if
	\begin{enumerate}
		\item there exists an $s$-$t$ disconnecting pair of order $k$ and size $l$ and
		
\item there is no $s$-$t$ disconnecting pair of cardinality less than $k+l$
having order at most $k$ and size at most $l$.\label{prop2_connectivity_pair}
	\end{enumerate}	
\end{definition}
\noindent As Property~\ref{prop2_connectivity_pair} implies, that there exist
$k$ vertices other than $s$ and $t$ and at least $l$ edges, we may replace
Property~\ref{prop2_connectivity_pair} by
\begin{align*}
	\begin{minipage}{.9\textwidth}
		\begin{itemize}
			\item[\ref{prop2_connectivity_pair}$'$] there is no $s$-$t$ 
			disconnecting pair of order $k$ and size $l-1$ or order $k-1$ and size $l$.
		\end{itemize}
	\end{minipage}
\end{align*}
Further, if there are fewer than $l$ edges between $s$ and $t$, then we can
replace Property~\ref{prop2_connectivity_pair} by
\begin{align*}
	\begin{minipage}{.9\textwidth}
		\begin{itemize}
			\item[\ref{prop2_connectivity_pair}$''$] there is no $s$-$t$ 
			disconnecting pair of order $k$ and size $l-1$
		\end{itemize}
	\end{minipage}
\end{align*}
This is true since we may replace any edge in an $s$-$t$ disconnecting pair by 
a vertex incident to it unless the edge joins $s$ and $t$.

The Beineke-Harary-Conjecture is, in some sense, a mixed version of Menger's 
Theorem. As we make use of it, we recall three versions of Menger's 
Theorem here.

\begin{theorem}[Menger's Theorem]\label{thm:mengers_theorem}
	Let $s$ and $t$ be two distinct vertices of a graph~$G$.
	\begin{thmenum}
		
		\item If $st\notin E(G)$, then the minimum number of vertices separating $s$
and $t$ in~$G$ is equal to the maximum number of internally disjoint
$s$-$t$ paths.\label{thm:menger_vertex}

\item The minimum number of edges separating $s$ and $t$ in $G$ is equal to the
maximum number of edge-disjoint $s$-$t$ paths in $G$.\label{thm:menger_edge}

\item The minimum cardinality of an $s$-$t$ disconnecting pair is equal to the
maximum number of internally disjoint $s$-$t$ paths.
\label{thm:menger_variant_adjacent_vertices}
	\end{thmenum}
\end{theorem}
\begin{proof}
	Proofs for the statements~\ref{thm:menger_vertex} and~\ref{thm:menger_edge} can
be found, for example, in~\cite{Wes01}. The
statement~\ref{thm:menger_variant_adjacent_vertices} is a direct consequence
of~\ref{thm:menger_vertex}: Any edge joining $s$ and $t$ induces an $s$-$t$
path that is internally disjoint to all other $s$-$t$ paths. Also every
edge joining $s$ and $t$ is contained in every $s$-$t$ disconnecting pair. The
statement now follows considering that any edge in an $s$-$t$ disconnecting
pair that does not join $s$ and $t$ can be replaced by one of its endvertices.
\end{proof}
Menger's Theorem implies the Beineke-Harary-Conjecture for a couple of base
cases regarding the integers $k$ and $l$.

\begin{observation}\label{obs:bineke_harary_l=1_or_k=0}
	Let $k\geq 0$ and $l\geq 1$ be integers and let $s$ and $t$ be two distinct 
	vertices of a graph~$G$.
	\begin{enumerate}
		\item If $(k,0)$ is a connectivity pair for $s$ and $t$, then $s$ and $t$ are
not adjacent. Further, the minimum number of vertices separating $s$ and $t$
is $k$ and, by Menger's Theorem, there exist $k$ internally disjoint $s$-$t$
paths.\label{obs:l=0}

\item If $(k,1)$ is a connectivity pair for $s$ and $t$, then $s$ and $t$ are
$k+1$ vertex-connected in $G$ and hence, by Menger's Theorem, there are $k+1$
internally disjoint paths between $s$ and $t$.\label{obs:l=1}

\item If $(0,l)$ is a connectivity pair for $s$ and $t$, then $s$ and $t$ are
$l$ edge-connected in $G$ and hence, by Menger's Theorem, there are $l$
edge-disjoint paths between $s$ and $t$.\label{obs:k=0}
	\end{enumerate}
\end{observation}

Another rather basic result implies that it suffices to prove the
Beineke-Harary-Conjecture for non-adjacent vertices as we see in the following two lemmas.

\begin{lemma}
	\label{lem:connectivity_pairs_with_st_edge}
	Let $G$ be a graph, $s,t\in V(G)$ be two distinct vertices and let $k,l$ be 
	non-negative integers. 
	The pair $(k,l)$ is a connectivity pair for $s$ and $t$ in $G$ if and only if 
	$(k,l-\card{E(s,t)})$ is a connectivity pair for $s$ and $t$ in $G-E(s,t)$.	
\end{lemma}
\begin{proof}
	Any $s$-$t$ disconnecting pair in $G$ has to contain all edges in 
	$E(s,t)$. Thus, 
	we get a one-to-one correspondence between the $s$-$t$ disconnecting pairs 
	in $G$ and the ones in $G-E(s,t)$ by mapping a pair $(W,F)$ to the pair 
	$(W,F\setminus E(s,t))$. The desired result follows immediately.
\end{proof}

\begin{lemma}
	\label{lem:conjecture_non_adjacent_suffices}
	Let $\mathcal{G}$ be a class of graphs which is closed under deletion of edges.
	If the Beineke-Harary-Conjecture holds for all graphs $G\in\mathcal{G}$ and
	all vertices $s,t\in V(G)$ such that $s$ and $t$ are not adjacent, then the
	conjecture holds for all graphs $G\in\mathcal{G}$ and all vertices $s,t\in V(G)$.
\end{lemma}
\begin{proof}
	Assume the Beineke-Harary-Conjecture holds for all graphs
	$G^\prime\in\mathcal{G}$ and all vertices $s,t\in V(G^\prime)$ with
	$\card{E(s,t)}=0$.	Let $G\in\mathcal{G}$ be a graph, $s,t\in V(G)$ distinct
	vertices with $\card{E(s,t)}\geq 1$, and let $(k,l)$ be a connectivity pair for
	$s$ and $t$ in $G$. By Lemma~\ref{lem:connectivity_pairs_with_st_edge},
	$(k,l-\card{E(s,t)})$ is a connectivity pair for $s$ and $t$ in $G-E(s,t)$.
	Thus, by assumption there exist $k+l-\card{E(s,t)}$ edge-disjoint $s$-$t$ paths
	of which at least $k$ are internally disjoint in $G-E(s,t)$. Note that
	we cannot assume that $k+1$ paths are internally disjoint, as
	$l-\card{E(s,t)}=0$ is a possibility. Nevertheless, the $k+l-\card{E(s,t)}$
	paths together with the edges in $E(s,t)$ yield $k+l$ edge-disjoint $s$-$t$
	paths of which at least $k+1$ are internally disjoint, as the edges in
	$E(s,t)$ are internally disjoint to all $s$-$t$ paths and by assumption
	$\card{E(s,t)}\geq 1$.
\end{proof}

Other than these simple observation the only meaningful result on the 
Beineke-Harary-Conjecture to date is due to Enomoto and Kaneko. Their result
from~\cite{enomoto1994} implies the correctness for further base cases regarding
the integers $k$ and $l$. We mention one explicit choice as a corollary, as we
make use of the statement later on.
\begin{theorem}[Enomoto and Kaneko\cite{enomoto1994}]\label{thm:enomoto_kaneko}
	Let $q$, $r$, $k$ and $l$ be integers with $k\geq 0$ and $l\geq 1$ such 
	that $k+l=q(k+1)+r$, $1\leq r\leq k+1$, and let $s$ and $t$ be distinct 
	vertices of a graph~$G$. If $q+r>k$ and if $(k,l)$ is a connectivity pair 
	for $s$ and $t$, then $G$ contains $k+l$ edge-disjoint $s$-$t$ paths of 
	which $k+1$ are internally disjoint. \qed
\end{theorem}

\begin{corollary}\label{cor::enomoto_kaneko_k=1}
	Let $(1,l)$ be a connectivity pair for two distinct vertices $s$ and $t$ of 
	a graph~$G$, then there are $l+1$ edge-disjoint $s$-$t$ paths of which two 
	are internally disjoint.
\end{corollary}

\begin{proof}
	For $l=1$ the statement holds due to Observation 
	\ref{obs:bineke_harary_l=1_or_k=0}. For $l\geq 2$ and
	$q,r\in\mathbb{N}$ with $1+l = q\cdot 2 + r$ and $1\leq r\leq 2$ we have
	$q+r>1$ and by Theorem~\ref{thm:enomoto_kaneko} we get the desired paths.
\end{proof}

Before we turn to the proof of the Beineke-Harary-Conjecture for $l=2$, we
discuss an erroneous claim made by Sadeghi and Fan in~\cite{sadeghi2019}. This
serves to illustrate the difficulties when trying to prove
Beineke-Harary-Conjecture and further shows why the conjecture does not claim
equivalence of the existence of connectivity pairs and paths. The statement by
Sadeghi and Fan is the following:

\emph{When $V(G)\geq k+l+1$, $k\geq 0$ and $l\geq 1$, a graph $G$ has $k+l$
		edge-disjoint paths of which $k+1$ are internally disjoint between any
		two vertices, if and only if the graph cannot be disconnected by removing $k$
		vertices and $l-1$ edges.}

\begin{figure}[t]
	\centering
	\begin{minipage}{.49\textwidth}
		\centering
		\begin{tikzpicture}[node distance=\nodedistance]
			\node[vertex, color=color2, fill=color2] (1) {};
			\node[vertex,above of=1, xshift=-\nodedistance] (2) {};
			\node[vertex, left of=1, xshift=-\nodedistance] (3) {};
			\node[vertex, below of=1, xshift=-\nodedistance] (4) {};
			\node[vertex, above of=1, xshift=\nodedistance] (5) {};
			\node[vertex, right of=1, xshift=\nodedistance] (6) {};
			\node[vertex, below of=1, xshift=\nodedistance] (7) {};
			
			\node[below of=1, yshift=.5\nodedistance] (1l) {$x_1$};
			\node[above of=2, yshift=-.5\nodedistance] (2l) {$x_2$};
			\node[left of=3, xshift=.5\nodedistance] (3l) {$x_3$};
			\node[below of=4, yshift=.5\nodedistance] (4l) {$x_4$};
			\node[above of=5, yshift=-.5\nodedistance] (5l) {$x_5$};
			\node[right of=6, xshift=-.5\nodedistance] (1l) {$x_6$};
			\node[below of=7, yshift=.5\nodedistance] (7l) {$x_7$};
			
			\foreach \x/\y in {1/2, 1/3, 1/4, 1/5, 1/6, 1/7, 2/3, 2/4, 3/4, 5/6, 5/7,
				6/7} {
				\drawEdge{\x}{\y}
			}
			\drawEdgeCol{2}{5}{color2}
		\end{tikzpicture}
	\end{minipage}
	\begin{minipage}{.49\textwidth}
		\centering
		\begin{tikzpicture}[node distance=\nodedistance]
			\node[vertex] (1) {};
			\node[vertex, above of=1, xshift=-\nodedistance] (2) {};
			\node[vertex, color=color4, fill=color4, left of=1, xshift=-\nodedistance]
			(3) {};
			\node[vertex, below of=1, xshift=-\nodedistance] (4) {};
			\node[vertex, above of=1, xshift=\nodedistance] (5) {};
			\node[vertex, right of=1, xshift=\nodedistance] (6) {};
			\node[vertex, color=color4, fill=color4, below of=1, xshift=\nodedistance]
			(7) {};
			
			\node[below of=1, yshift=.5\nodedistance] (1l) {$x_1$};
			\node[above of=2, yshift=-.5\nodedistance] (2l) {$x_2$};
			\node[left of=3, xshift=.5\nodedistance] (3l) {$x_3$};
			\node[below of=4, yshift=.5\nodedistance] (4l) {$x_4$};
			\node[above of=5, yshift=-.5\nodedistance] (5l) {$x_5$};
			\node[right of=6, xshift=-.5\nodedistance] (1l) {$x_6$};
			\node[below of=7, yshift=.5\nodedistance] (7l) {$x_7$};
			
			\foreach \x/\y in {1/2, 1/5, 1/6, 1/7, 2/4, 5/6} {
				\drawEdge{\x}{\y}
			}
			\foreach \x/\y in {3/2, 2/5, 5/7} { \drawEdgeCol{\x}{\y}{color1} }
			\foreach \x/\y in {3/1,1/7} {\drawEdgeCol{\x}{\y}{color3}}
			\foreach \x/\y in {3/4, 4/1, 1/6, 6/7}{\drawEdgeCol{\x}{\y}{color5}}
		\end{tikzpicture}
	\end{minipage}
	\caption{A graph containing a vertex-edge separator, such that between any pair
		of
		vertices there exist three edge-disjoint paths of which two are internally
		disjoint.}
	\label{fig:counter}
\end{figure}
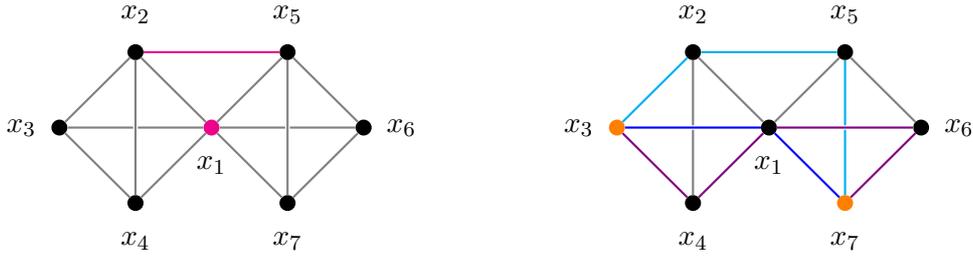

In~\cite{sadeghi2019} for integers $k,l\geq 1$ a graph $G$ with at least $k+l+1$
vertices is called \emph{$(k,l)$-connected} if it cannot be disconnected by
removing $k$~vertices and $l-1$~ edges. The following claim is then made.
\begin{align}
	\label{claim:sadeghi_1}
	\begin{minipage}{.9\textwidth}
		Let $k,l\geq 1$ and $G$ be a graph with at least
		$k+l+1$~vertices. Then~$G$ is $(k,l)$-connected if and only if $G$ is
		$k+1$~vertex-connected and $k+l$~edge-connected.
	\end{minipage}
\end{align}
If $G$ is in fact $(k,l)$-connected it can readily be observed that it is also
$k+1$ vertex-connected and $k+l$ edge-connected. On the other hand $G$ being
$k+1$ vertex-connected and $k+l$ edge-connected does \emph{not} imply
$(k,l)$-connectivity.  To see this, consider the two complete graphs $G_1$ and
$G_2$ on the vertex sets $\{x_1,x_2,x_3,x_4\}$ and $\{x_1, x_5, x_6, x_7\}$. We
construct a graph~$G$ by regarding the union of $G_1$ and $G_2$ and
additionally adding an edge between vertices~$x_5$ and~$x_2$.
Figure~\ref{fig:counter} displays the constructed graph. The graph $G$ is
$2$-vertex-connected and $3$-edge-connected, but it is not $(1,2)$-connected as
the removal of the vertex~$x_1$ and the edge $x_2x_5$ disconnects the graph.
Thus, the Claim~\eqref{claim:sadeghi_1} cannot hold. As a corollary
of Claim~\eqref{claim:sadeghi_1}, Sadeghi and Fan state the following.
\begin{align}
	\label{claim:sadeghi_2}  
	\begin{minipage}{.9\textwidth}
		Let $k\geq 0$, $l\geq 1$, and $G$ be a graph with
		at least $k+l+1$~vertices. Then $G$ is $(k,l)$-connected if and only if it has
		$k+l$ edge-disjoint paths between every pair of vertices of which $k+1$~paths
		are internally disjoint.
	\end{minipage}
\end{align}
As a corollary to Claim~\eqref{claim:sadeghi_1},
Claim~\eqref{claim:sadeghi_2} cannot be considered proven. We give a
counterexample to the claim in Proposition~\ref{prop:counter}.

In the original conjecture by Beineke and Harary~\cite{B01} and in the
extension due to~\cite{enomoto1994} it is never claimed that the existence of the
desired paths is sufficient for $(k,l)$-connectivity and, in fact, it is
not. For the sake of completeness we argue why the existence of the paths in
Claim~\eqref{claim:sadeghi_2} is not sufficient.

\begin{proposition}
	\label{prop:counter} The graph $G$ constructed above contains a separator
	of one vertex and one edge and between any pair of vertices there exist three
	edge-disjoint paths of which two are internally disjoint.
\end{proposition}
\begin{proof}
	Consider the graph $G$ above, that also provided a counterexample to
Claim~\eqref{claim:sadeghi_1}, see Figure~\ref{fig:counter}. The vertex $x_1$
together with the edge $x_5x_2$ disconnects the graph. Now let $v_1,v_2\in
V(G)$. If $v_1,v_2\in V(G_i)$ for some $i\in\{1,2\}$, then there are three internally
disjoint $v_1$-$v_2$ paths. Otherwise, without loss of generality $v_1\in
\{x_2,x_3,x_4\}$ and $v_2\in\{x_5,x_6,x_7\}$. Denote by $P_1$ a shortest path
from $v_1$ to $x_2$ (This is either a single edge or the path without edges)
and by $P_2$ a shortest path from~$x_5$ to~$v_2$. We define the $v_1$-$v_2$
path $P\coloneqq (P_1\cup P_2)+x_2x_5$.  Further, let $Q=v_1x_1v_2$. Finally, let
$w_1\in\{x_3,x_4\}\setminus\{v_1\}$ and $w_2\in\{x_6,x_7\}\setminus\{v_2\}$ and
define the path $R=v_1w_1x_1w_2v_2$. It is easily verified that $P$, $Q$, $R$
are three edge-disjoint $v_1$-$v_2$ paths and $P$ and $Q$ are also internally
disjoint.
\end{proof}

The graph in Figure~\ref{fig:counter} illustrates two things. On the one
hand it shows that we may not hope to prove an equivalence in the fashion of
Claim~\eqref{claim:sadeghi_2}. On the other hand it shows that it is not
possible to replace the mixed form of connectivity by two separate statements on
pure connectivity in the fashion of Claim~\eqref{claim:sadeghi_1}. This is
one of the reasons why the Beineke-Harary-Conjecture is not a consequence of
Menger's Theorem and its proof has not been established as of yet. It also
suggests that the usual techniques used for proofs of Menger's Theorem might not
transfer to the mixed statement. In the following we use a novel technique for
proving the Beineke-Harary-Conjecture for the case that $l=2$. The idea is to
keep the desired $k+1$ internally disjoint paths and move from $s$ to $t$
along the remaining path. The statement is then proved by induction.

\section{The Beineke-Harary-Conjecture for Disconnecting Pairs of Size 2}
\label{sec:main_result}
Now that we have established some foundations for connectivity pairs and the
Beineke-Harary-Conjecture, we turn to the main result of this contribution. We
prove that the Beineke-Harary-Conjecture holds for all non-negative
integers $k$ if $l=2$.

\begin{theorem}
	\label{the:main_result}
	Let $G$ be a graph and $s,t\in V(G)$. Further, let $(k,2)$ be a connectivity 
	pair for $s$ and $t$. Then, there exist $k+2$ edge-disjoint $s$-$t$ paths 
	of which $k+1$ are internally disjoint.	
\end{theorem}

Before we begin with the proof, note that the result of
Theorem~\ref{the:main_result} has only been proved for $k=1$. In particular, the
result of Enomoto and Kaneko, \textit{\textit{cf.}}~Theorem~\ref{thm:enomoto_kaneko},
basically tackles the conjecture from a different angle: In their statement for
$k\geq 2$ and $l=2$, the sum $q+r$ always equals $2$, which leads to a large gap
between $k$ and $q+r$ for large $k$.

The idea of the proof of Theorem~\ref{the:main_result} for vertices $s_1$ and
$t$ is to always keep an \skein{s_1}{t}{(k+1)} and inductively move along some
other $s_1$-$t$ path $P$ which is edge-disjoint to the currently regarded
\skein{s}{t}{(k+1)}.
In order to use induction we generalize the claim of Theorem~\ref{the:main_result} and expand in the current step an $s_2$-$t$ path that is edge-disjoint to the currently regarded \skein{s_1}{t}{(k+1)} until we obtain $P$ in the end.


\begin{samepage}
\begin{theorem}
	\label{the:main}
	Let $G$ be a graph, $s_1,s_2,t\in V(G)$ with $s_1\neq t$. Further, assume 
	that 
	\begin{enumerate}
		\item there exists an $s_2$-$t$ path in $G$,\label{prop:1}
		\item there exists an \skein{s_1}{t}{(k+1)} in $G$,
		and \label{prop:2}
		\item there is no $\{s_1,s_2\}$-t disconnecting pair of cardinality 
		$k+1$ and order at most $k$ in~$G$.\label{prop:3}
	\end{enumerate}
	Then, there exist $k+2$ edge-disjoint paths, of which $k+1$ are internally 
	disjoint $s_1$-$t$ paths and one is an $s_2$-$t$ path.
\end{theorem}
\end{samepage}

\begin{proof}
	Let $G$ be a graph, $s_1,s_2,t\in V(G)$ with $s_1\neq t$ satisfying
	Properties~\ref{prop:1} to~\ref{prop:3}. We prove the
	claim by induction on the number of edges $\lvert E(G)\rvert$. If $\lvert
	E(G)\rvert\leq k$, then there cannot be $k+1$ internally disjoint $s_1$-$t$
	paths, as $s_1\neq t$. Thus, from now on we may assume the following.
	\begin{align}
		\label{gt::thm:bhc_l2_helper_sub1}
		\begin{minipage}{.9\textwidth}
			\textit{Let $G^\prime$ be a graph with $\card{E(G^\prime)}<\card{E(G)}$ and
				vertices $s_1^\prime, s_2^\prime, t^\prime\in V(G^\prime)$ with
				$s_1^\prime\neq
				t^\prime$. If Properties~\ref{prop:1} to~\ref{prop:3} are
				satisfied in $G^\prime$, then there exist $k+2$ edge-disjoint paths of which
				$k+1$ are internally disjoint $s_1^\prime$-$t^\prime$ paths and of
				which
				one is an $s_2^\prime$-$t^\prime$ path.}
		\end{minipage} 
	\end{align}
	We begin by proving the induction step for the case that $s_2$ is contained in
	an \skein{s_1}{t}{(k+1)} and afterwards use this result to prove the induction
	step for the case that $s_2$ is not contained in such a skein.
	
	\noindent\textbf{Case 1:} The vertex $s_2$ is contained in an
	\skein{s_1}{t}{(k+1)}.
	
	If $s_2=t$, then the $k+1$ internally disjoint paths from
	Property~\ref{prop:2} together with the $s_2$-$t$ path $s_2=t$ form the desired
	paths. Thus, we may assume that ${s_2\neq t}$. Denote by $P_1,\dots, P_{k+1}$ the
	$s_1$-$t$ paths of an \skein{s_1}{t}{(k+1)} containing~$s_2$. Without loss of
	generality we may assume $s_2\in V(P_{k+1})$. Denote by $s_2^\prime$ the vertex
	succeeding $s_2$ on $P_{k+1}$, i.\ e.\ $P_{k+1}=s_1\dots s_2s_2^\prime \dots t$, \textit{cf.}~Figure~\ref{gt::fig:stskeins_s2on}.
	Note that $s_1=s_2$ is not forbidden at this point. We now want to use the
	induction hypothesis for $G-s_2s_2^\prime$ and the vertices $s_1$, $s_2^\prime$
	and $t$, \textit{cf.}~Figure~\ref{gt::fig:stsep_s2on} a).
	
	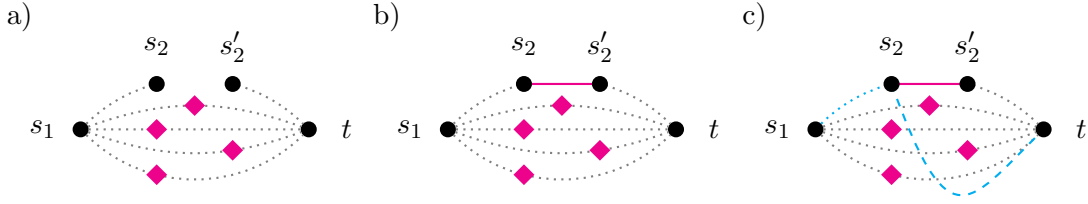
\begin{figure}[ht]
		\begin{minipage}[t]{.32\textwidth}%
			\centering
			\begin{tikzpicture}[node distance=\nodedistance]
				\clip(-1,1.7) rectangle (3.6, -1);
				\node at (-0.8,1.5) (x) {a)};
				
				\node[vertex] (s1) {};
				\labelLeft{s1}{$s_1$};
				\node[vertex, xshift=3\nodedistance] (t) {};
				\labelRight{t}{$t$};
				
				\path[name path= p1] (s1) to[out=45, in=135] (t); 
				
				\coordinate[right of=s1, yshift=\nodedistance] (c1);
				\coordinate[right of=s1, yshift=-\nodedistance] (c2);
				\path[name path=leftmid] (c1) -- (c2);
				\coordinate[left of=t, yshift=\nodedistance] (c3);
				\coordinate[left of=t, yshift=-\nodedistance] (c4);
				\path[name path=rightmid] (c3) -- (c4);
				\path[name intersections={of=leftmid and p1}];
				\coordinate (s2) at (intersection-1);
				\labelAbove{s2}{$s_2$}
				\path[name intersections={of=rightmid and p1}];
				\coordinate (s2p) at (intersection-1);
				\labelAbove{s2p}{$s_2^\prime$}
				\begin{scope}
					\clip(s1)rectangle(s2)(s2p)rectangle(t);
					\draw[edge, dotted] (s1) to[out=45, in=135] (t);
				\end{scope}
				\draw[name path=l1, edge, dotted] (s1) to[out=20, in=160] (t);
				\draw[name path=l2, edge, dotted] (s1) to (t);
				\draw[name path=l3, edge, dotted] (s1) to[out=-20, in=200] (t);
				\draw[name path=l4, edge, dotted] (s1) to[out=-45, in=225] (t);
				\node[vertex] at (s2) (s2v) {};
				\node[vertex] at (s2p) (s2pv) {};
				\coordinate[right of=s1, xshift=.5\nodedistance, yshift=\nodedistance] (m1);
				\coordinate[right of=s1, xshift=.5\nodedistance, yshift=-\nodedistance]
				(m2);
				\path[name path=mid-l1] (m1) -- (m2);
				\path[name path=mid-l2] (c1) -- (c2);
				\path[name path=mid-l3] (c3) -- (c4);
				\path[name path=mid-l4] (c1) -- (c2);
				\foreach \x in {l1, l2, l3, l4} {
					\path[name intersections={of={mid-\x} and \x}];
					\node[diamond, fill=color2,inner sep=2pt] at (intersection-1) (sep-\x) {}; 
				}
				
			\end{tikzpicture}
		\end{minipage}
		\begin{minipage}[t]{.32\textwidth}%
			\centering
			\begin{tikzpicture}[node distance=\nodedistance]
				\clip(-1,1.7) rectangle (3.6, -1);
				\node at (-0.8,1.5) (x) {b)};
				
				\node[vertex] (s1) {};
				\labelLeft{s1}{$s_1$};
				\node[vertex, xshift=3\nodedistance] (t) {};
				\labelRight{t}{$t$};
				
				\path[name path= p1] (s1) to[out=45, in=135] (t); 
				
				\coordinate[right of=s1, yshift=\nodedistance] (c1);
				\coordinate[right of=s1, yshift=-\nodedistance] (c2);
				\path[name path=leftmid] (c1) -- (c2);
				\coordinate[left of=t, yshift=\nodedistance] (c3);
				\coordinate[left of=t, yshift=-\nodedistance] (c4);
				\path[name path=rightmid] (c3) -- (c4);
				\path[name intersections={of=leftmid and p1}];
				\coordinate (s2) at (intersection-1);
				\labelAbove{s2}{$s_2$}
				\path[name intersections={of=rightmid and p1}];
				\coordinate (s2p) at (intersection-1);
				\labelAbove{s2p}{$s_2^\prime$}
				\draw[edge, color2] (s2) to (s2p);
				\begin{scope}
					\clip(s1)rectangle(s2)(s2p)rectangle(t);
					\draw[edge, dotted] (s1) to[out=45, in=135] (t);
				\end{scope}
				\draw[name path=l1, edge, dotted] (s1) to[out=20, in=160] (t);
				\draw[name path=l2, edge, dotted] (s1) to (t);
				\draw[name path=l3, edge, dotted] (s1) to[out=-20, in=200] (t);
				\draw[name path=l4, edge, dotted] (s1) to[out=-45, in=225] (t);
				\node[vertex] at (s2) (s2v) {};
				\node[vertex] at (s2p) (s2pv) {};
				\coordinate[right of=s1, xshift=.5\nodedistance, yshift=\nodedistance] (m1);
				\coordinate[right of=s1, xshift=.5\nodedistance, yshift=-\nodedistance]
				(m2);
				\path[name path=mid-l1] (m1) -- (m2);
				\path[name path=mid-l2] (c1) -- (c2);
				\path[name path=mid-l3] (c3) -- (c4);
				\path[name path=mid-l4] (c1) -- (c2);
				\foreach \x in {l1, l2, l3, l4} {
					\path[name intersections={of={mid-\x} and \x}];
					\node[diamond, fill=color2,inner sep=2pt] at (intersection-1) (sep-\x) {}; 
				}
				
			\end{tikzpicture}
		\end{minipage}
		\begin{minipage}[t]{.32\textwidth}%
			\centering
			\begin{tikzpicture}[node distance=\nodedistance]
				\clip(-1,1.7) rectangle (3.6, -1);
				\node at (-0.8,1.5) (x) {c)};
				
				\node[vertex] (s1) {};
				\labelLeft{s1}{$s_1$};
				\node[vertex, xshift=3\nodedistance] (t) {};
				\labelRight{t}{$t$};
				
				\path[name path= p1] (s1) to[out=45, in=135] (t); 
				
				\coordinate[right of=s1, yshift=\nodedistance] (c1);
				\coordinate[right of=s1, yshift=-\nodedistance] (c2);
				\path[name path=leftmid] (c1) -- (c2);
				\coordinate[left of=t, yshift=\nodedistance] (c3);
				\coordinate[left of=t, yshift=-\nodedistance] (c4);
				\path[name path=rightmid] (c3) -- (c4);
				\path[name intersections={of=leftmid and p1}];
				\coordinate (s2) at (intersection-1);
				\labelAbove{s2}{$s_2$}
				\path[name intersections={of=rightmid and p1}];
				\coordinate (s2p) at (intersection-1);
				\labelAbove{s2p}{$s_2^\prime$}
				\draw[edge, color2] (s2) to (s2p);
				\begin{scope}
					\clip(s1)rectangle(s2);
					\draw[edge, dotted, color1] (s1) to[out=45, in=135] (t);
				\end{scope}
				\begin{scope}
					\clip(s2p)rectangle(t);
					\draw[edge, dotted] (s1) to[out=45, in=135] (t);
				\end{scope}
				\draw[name path=l1, edge, dotted] (s1) to[out=20, in=160] (t);
				\draw[name path=l2, edge, dotted] (s1) to (t);
				\draw[name path=l3, edge, dotted] (s1) to[out=-20, in=200] (t);
				\draw[name path=l4, edge, dotted] (s1) to[out=-45, in=225] (t);
				
				\coordinate[right of=s1, xshift=.5\nodedistance, yshift=\nodedistance] (m1);
				\coordinate[right of=s1, xshift=.5\nodedistance, yshift=-\nodedistance]
				(m2);
				\path[name path=mid-l1] (m1) -- (m2);
				\path[name path=mid-l2] (c1) -- (c2);
				\path[name path=mid-l3] (c3) -- (c4);
				\path[name path=mid-l4] (c1) -- (c2);
				\foreach \x in {l1, l2, l3, l4} {
					\path[name intersections={of={mid-\x} and \x}];
					\node[diamond, fill=color2,inner sep=2pt] at (intersection-1) (sep-\x) {}; 
				}
				
				\coordinate[right of=c2, xshift=-.25\nodedistance, yshift=-.5\nodedistance]
				(c2p);
				\draw[edge, dashed, color1] (s2) .. controls (c2p) and (c4) .. (t);
				\node[vertex] at (s2) (s2v) {};
				\node[vertex] at (s2p) (s2pv) {};
			\end{tikzpicture}
		\end{minipage}
		
		\caption{Case~1 in the proof of Theorem~\ref{the:main}: Supposed separation of
	$s_1$ and~$t$. The colored diamonds correspond to elements in $(W,F)$. The
	dotted lines are mutually internally disjoint. The solid line is a single
	edge. The colored lines that are not solid indicate a connection of vertices that does not touch
	colored vertices or edges.} \label{gt::fig:stsep_s2on}
	\end{figure}

	Property~\ref{prop:1} is satisfied as $s_2^\prime P_{k+1}$ is an
	$s_2^\prime$-$t$ path in $G-s_2s_2^\prime$. Suppose that there do not exist
	$k+1$ internally disjoint $s_1$-$t$ paths in $G-s_2s_2^\prime$. By
	Menger's Theorem there is an $s_1$-$t$ disconnecting pair $(W,F)$ of
	cardinality $k$. Since in $G-s_2s_2^\prime$ the internally disjoint
	paths $P_1,\dots, P_k$ still exist, all elements of $(W,F)$ are contained in the
	paths $P_1,\dots, P_k$, \textit{cf.}~Figure~\ref{gt::fig:stsep_s2on} a). Thus, the path
	$P_{k+1}s_2$ still exists in
	$G-s_2s_2^\prime-W-F$ and $(W,F)$ is an $\{s_1,s_2\}$-$t$ disconnecting pair in
	$G-s_2s_2^\prime$, \textit{cf.}~Figure~\ref{gt::fig:stsep_s2on} b). By assumption $(W,
	F\cup\{s_2s_2^\prime\})$ is not an
	$\{s_1,s_2\}$-$t$ disconnecting pair in $G$ and there exists some
	$\{s_1,s_2\}$-$t$ path in $G-W-F-s_2s_2^\prime$,
	\textit{cf.}~Figure~\ref{gt::fig:stsep_s2on} c), which yields a contradiction. Hence,
	Property~\ref{prop:2}
	is satisfied in $G-s_2s_2^\prime$ and $s_1,s_2^\prime, t$.
	
	\begin{figure}[ht]
		\begin{minipage}[t]{.32\textwidth}%
			\centering
			\begin{tikzpicture}[node distance=\nodedistance]
			\clip(-1,1.7) rectangle (3.6, -1);
			\node at (-0.8,1.5) (x) {a)};
			
			\node[vertex] (s1) {};
			\labelLeft{s1}{$s_1$};
			\node[vertex, xshift=3\nodedistance] (t) {};
			\labelRight{t}{$t$};
			
			\path[name path= l5] (s1) to[out=45, in=135] (t); 
			
			\coordinate[right of=s1, yshift=\nodedistance] (c1);
			\coordinate[right of=s1, yshift=-\nodedistance] (c2);
			\path[name path=leftmid] (c1) -- (c2);
			\coordinate[left of=t, yshift=\nodedistance] (c3);
			\coordinate[left of=t, yshift=-\nodedistance] (c4);
			\path[name path=rightmid] (c3) -- (c4);
			\path[name intersections={of=leftmid and l5}];
			\coordinate (s2) at (intersection-1);
			\labelAbove{s2}{$s_2$}
			\path[name intersections={of=rightmid and l5}];
			\coordinate (s2p) at (intersection-1);
			\labelAbove{s2p}{$s_2^\prime$}
			\begin{scope}
			\clip(s1)rectangle(s2)(s2p)rectangle(t);
			\draw[edge, dotted] (s1) to[out=45, in=135] (t);
			\end{scope}
			\draw[name path=l1, edge, dotted] (s1) to[out=20, in=160] (t);
			\draw[name path=l2, edge, dotted] (s1) to (t);
			\draw[name path=l3, edge, dotted] (s1) to[out=-20, in=200] (t);
			\draw[name path=l4, edge, dotted] (s1) to[out=-45, in=225] (t);
			\node[vertex] at (s2) (s2v) {};
			\node[vertex] at (s2p) (s2pv) {};
			\coordinate[right of=s1, xshift=.5\nodedistance, yshift=\nodedistance] (m1);
			\coordinate[right of=s1, xshift=.5\nodedistance, yshift=-\nodedistance]
			(m2);
			\path[name path=mid-l1] (m1) -- (m2);
			\path[name path=mid-l2] (c1) -- (c2);
			\path[name path=mid-l3] (c3) -- (c4);
			\path[name path=mid-l4] (c1) -- (c2);
			\coordinate[left of=t, xshift=.5\nodedistance, yshift=\nodedistance] (c5);
			\coordinate[left of=t, xshift=.5\nodedistance, yshift=-\nodedistance] (c6);
			\path[name path=mid-l5] (c5) -- (c6);
			\foreach \x in {l1, l2, l3, l4, l5} {
				\path[name intersections={of={mid-\x} and \x}];
				\node[diamond, fill=color2,inner sep=2pt] at (intersection-1) (sep-\x) {}; 
			}
			\end{tikzpicture}
		\end{minipage}
		\begin{minipage}[t]{.32\textwidth}%
			\centering
			\begin{tikzpicture}[node distance=\nodedistance]
			\clip(-1,1.7) rectangle (3.6, -1);
			\node at (-0.8,1.5) (x) {b)};
			
			\node[vertex] (s1) {};
			\labelLeft{s1}{$s_1$};
			\node[vertex, xshift=3\nodedistance] (t) {};
			\labelRight{t}{$t$};
			
			\path[name path= l5] (s1) to[out=45, in=135] (t); 
			
			\coordinate[right of=s1, yshift=\nodedistance] (c1);
			\coordinate[right of=s1, yshift=-\nodedistance] (c2);
			\path[name path=leftmid] (c1) -- (c2);
			\coordinate[left of=t, yshift=\nodedistance] (c3);
			\coordinate[left of=t, yshift=-\nodedistance] (c4);
			\path[name path=rightmid] (c3) -- (c4);
			\path[name intersections={of=leftmid and l5}];
			\coordinate (s2) at (intersection-1);
			\labelAbove{s2}{$s_2$}
			\path[name intersections={of=rightmid and l5}];
			\coordinate (s2p) at (intersection-1);
			\labelAbove{s2p}{$s_2^\prime$}
			\draw[edge] (s2) to (s2p);
			\begin{scope}
			\clip(s1)rectangle(s2)(s2p)rectangle(t);
			\draw[edge, dotted] (s1) to[out=45, in=135] (t);
			\end{scope}
			\draw[name path=l1, edge, dotted] (s1) to[out=20, in=160] (t);
			\draw[name path=l2, edge, dotted] (s1) to (t);
			\draw[name path=l3, edge, dotted] (s1) to[out=-20, in=200] (t);
			\draw[name path=l4, edge, dotted] (s1) to[out=-45, in=225] (t);
			\node[vertex] at (s2) (s2v) {};
			\node[vertex] at (s2p) (s2pv) {};
			\coordinate[right of=s1, xshift=.5\nodedistance, yshift=\nodedistance] (m1);
			\coordinate[right of=s1, xshift=.5\nodedistance, yshift=-\nodedistance]
			(m2);
			\path[name path=mid-l1] (m1) -- (m2);
			\path[name path=mid-l2] (c1) -- (c2);
			\path[name path=mid-l3] (c3) -- (c4);
			\path[name path=mid-l4] (c1) -- (c2);
			\coordinate[left of=t, xshift=.5\nodedistance, yshift=\nodedistance] (c5);
			\coordinate[left of=t, xshift=.5\nodedistance, yshift=-\nodedistance] (c6);
			\path[name path=mid-l5] (c5) -- (c6);
			\foreach \x in {l1, l2, l3, l4, l5} {
				\path[name intersections={of={mid-\x} and \x}];
				\node[diamond, fill=color2,inner sep=2pt] at (intersection-1) (sep-\x) {}; 
			}
			
			\end{tikzpicture}
		\end{minipage}
		\begin{minipage}[t]{.32\textwidth}%
			\centering
			\begin{tikzpicture}[node distance=\nodedistance]
			\clip(-1,1.7) rectangle (3.6, -1);
			\node at (-0.8,1.5) (x) {c)};
			
			\node[vertex] (s1) {};
			\labelLeft{s1}{$s_1$};
			\node[vertex, xshift=3\nodedistance] (t) {};
			\labelRight{t}{$t$};
			
			\path[name path=l5] (s1) to[out=45, in=135] (t); 
			
			\coordinate[right of=s1, yshift=\nodedistance] (c1);
			\coordinate[right of=s1, yshift=-\nodedistance] (c2);
			\path[name path=leftmid] (c1) -- (c2);
			\coordinate[left of=t, yshift=\nodedistance] (c3);
			\coordinate[left of=t, yshift=-\nodedistance] (c4);
			\path[name path=rightmid] (c3) -- (c4);
			\path[name intersections={of=leftmid and l5}];
			\coordinate (s2) at (intersection-1);
			\labelAbove{s2}{$s_2$}
			\path[name intersections={of=rightmid and l5}];
			\coordinate (s2p) at (intersection-1);
			\labelAbove{s2p}{$s_2^\prime$}
			\draw[edge, color1] (s2) to (s2p);
			\begin{scope}
			\clip(s1)rectangle(s2);
			\draw[edge, dotted, color1] (s1) to[out=45, in=135] (t);
			\end{scope}
			\begin{scope}
			\clip(s2p)rectangle(t);
			\draw[edge, dotted] (s1) to[out=45, in=135] (t);
			\end{scope}
			\draw[name path=l1, edge, dotted] (s1) to[out=20, in=160] (t);
			\draw[name path=l2, edge, dotted] (s1) to (t);
			\draw[name path=l3, edge, dotted] (s1) to[out=-20, in=200] (t);
			\draw[name path=l4, edge, dotted] (s1) to[out=-45, in=225] (t);
			
			\coordinate[right of=s1, xshift=.5\nodedistance, yshift=\nodedistance] (m1);
			\coordinate[right of=s1, xshift=.5\nodedistance, yshift=-\nodedistance]
			(m2);
			\path[name path=mid-l1] (m1) -- (m2);
			\path[name path=mid-l2] (c1) -- (c2);
			\path[name path=mid-l3] (c3) -- (c4);
			\path[name path=mid-l4] (c1) -- (c2);
			\coordinate[left of=t, xshift=.5\nodedistance, yshift=\nodedistance] (c5);
			\coordinate[left of=t, xshift=.5\nodedistance, yshift=-\nodedistance] (c6);
			\path[name path=mid-l5] (c5) -- (c6);
			\foreach \x in {l1, l2, l3, l4, l5} {
				\path[name intersections={of={mid-\x} and \x}];
				\node[diamond, fill=color2,inner sep=2pt] at (intersection-1) (sep-\x) {}; 
			}
			
			\coordinate[right of=c2, xshift=-.25\nodedistance, yshift=-.5\nodedistance]
			(c2p);
			\draw[edge, dashed, color1] (s2) .. controls (c2p) and (c4) .. (t);
			\node[vertex] at (s2) (s2v) {};
			\node[vertex] at (s2p) (s2pv) {};
			\end{tikzpicture}
		\end{minipage}
		\caption{Case~1 in the proof of Theorem~\ref{the:main}: Supposed
			separation of $\{s_1,s_2^\prime\}$ and $t$. The colored diamonds correspond
			to elements in $(W,F)$. The dotted lines are mutually internally
			disjoint. The solid
			line is a single edge. The colored lines indicate a connection of vertices
			that
			does not touch $(W,F)$.} \label{gt::fig:stsep_s2onb}
	\end{figure}
	
	Now suppose there
	exists an $\{s_1,s_2^\prime\}$-$t$ disconnecting pair $(W,F)$ of cardinality
	$k+1$ and order at most $k$ in $G-s_2s_2^\prime$. As the paths
	$P_1,\dots,P_k,s_2^\prime P_{k+1}$ are internally disjoint, each element
	of $(W,F)$ is contained in one of these paths,
	\textit{cf.}~Figure~\ref{gt::fig:stsep_s2onb} a). Thus, $P_{k+1}s_2$ still exists
	in $G-s_2s_2^\prime-W- F$ and neither $s_2$ nor $s_2^\prime$ are contained in
	the same component as~$t$ in $G-s_2s_2^\prime-W -F$. This implies that $(W,F)$
	is an $\{s_1,s_2\}$-$t$ disconnecting pair in~$G$,
	\textit{cf.}~Figure~\ref{gt::fig:stsep_s2onb} b). Again this is a
	contradiction to Property~\ref{prop:3} in~$G$,
	\textit{cf.}~Figure~\ref{gt::fig:stsep_s2onb} c) and hence
	Property~\ref{prop:3} is satisfied for $G-s_2s_2^\prime$ and
	$s_1,s_2^\prime, t$.

	As $G-s_2s_2^\prime$ contains $\card{E(G)}-1$ edges,
	statement~\eqref{gt::thm:bhc_l2_helper_sub1} is applicable
	and there exist $k+2$ edge-disjoint paths of which $k+1$ are internally
	disjoint $s_1$-$t$ paths, say $P_1^\prime,\dots, P_{k+1}^\prime$, and of
	which one is an $s_2^\prime$-$t$ path, say $P_{k+2}^\prime$,
	\textit{cf.}~Figure~\ref{gt::fig:stskeins_s2on}~b). 
	
	If $s_2\in V(P_{k+2}^\prime)$ the paths $P_1^\prime,\dots, P_{k+1}^\prime,
	s_2P_{k+2}^\prime$ are the desired paths in $G$. Otherwise the paths
	$P_1^\prime,\dots, P_{k+1}^\prime,s_2s_2^\prime \cup P_{k+2}^\prime$ form the
	desired paths, \textit{cf.}~Figure~\ref{gt::fig:stskeins_s2on}~c).
	
	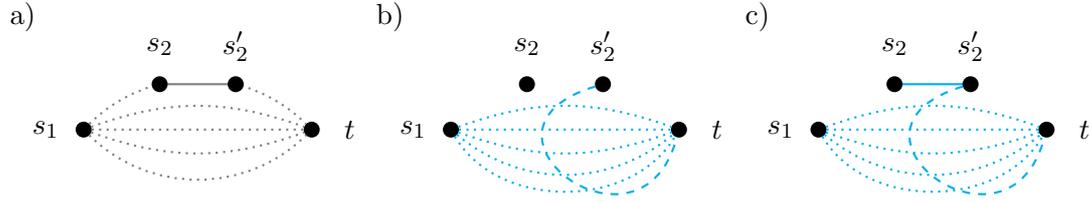
\begin{figure}[ht]
		\begin{minipage}[t]{.32\textwidth}%
			\centering
			\begin{tikzpicture}[node distance=\nodedistance]
			\clip(-1,1.7) rectangle (3.6, -1);
			\node at (-0.8,1.5) (x) {a)};
			
			\node[vertex] (s1) {};
			\labelLeft{s1}{$s_1$};
			\node[vertex, xshift=3\nodedistance] (t) {};
			\labelRight{t}{$t$};
			
			\path[name path= p1] (s1) to[out=45, in=135] (t); 
			
			\coordinate[right of=s1, yshift=\nodedistance] (c1);
			\coordinate[right of=s1] (c2);
			\path[name path=leftmid] (c1) -- (c2);
			\coordinate[left of=t, yshift=\nodedistance] (c3);
			\coordinate[left of=t] (c4);
			\path[name path=rightmid] (c3) -- (c4);
			\path[name intersections={of=leftmid and p1}];
			\coordinate (s2) at (intersection-1);
			\labelAbove{s2}{$s_2$}
			\path[name intersections={of=rightmid and p1}];
			\coordinate (s2p) at (intersection-1);
			\labelAbove{s2p}{$s_2^\prime$}
			\draw[edge] (s2) to (s2p);
			\begin{scope}
			\clip(s1)rectangle(s2)(s2p)rectangle(t);
			\draw[edge, dotted] (s1) to[out=45, in=135] (t);
			\end{scope}
			\draw[edge, dotted] (s1) to[out=20, in=160] (t);
			\draw[edge, dotted] (s1) to (t);
			\draw[edge, dotted] (s1) to[out=-20, in=200] (t);
			\draw[edge, dotted] (s1) to[out=-45, in=225] (t);
			\node[vertex] at (s2) (s2v) {};
			\node[vertex] at (s2p) (s2pv) {};
			\end{tikzpicture}
		\end{minipage}
		\begin{minipage}[t]{.32\textwidth}%
			\centering
			\begin{tikzpicture}[node distance=\nodedistance]
			\clip(-1,1.7) rectangle (3.6, -1);
			\node at (-0.8,1.5) (x) {b)};
			
			\node[vertex] (s1) {};
			\labelLeft{s1}{$s_1$};
			\node[vertex, xshift=3\nodedistance] (t) {};
			\labelRight{t}{$t$};
			
			\path[name path= p1] (s1) to[out=45, in=135] (t); 
			
			\coordinate[right of=s1, yshift=\nodedistance] (c1);
			\coordinate[right of=s1] (c2);
			\path[name path=leftmid] (c1) -- (c2);
			\coordinate[left of=t, yshift=\nodedistance] (c3);
			\coordinate[left of=t] (c4);
			\path[name path=rightmid] (c3) -- (c4);
			\path[name intersections={of=leftmid and p1}];
			\coordinate (s2) at (intersection-1);
			\labelAbove{s2}{$s_2$}
			\path[name intersections={of=rightmid and p1}];
			\coordinate (s2p) at (intersection-1);
			\labelAbove{s2p}{$s_2^\prime$}
			\node[vertex] at (s2) (s2v) {};
			\node[vertex] at (s2p) (s2pv) {};
			
			\draw[edge, dotted, color1] (s1) to[out=20, in=160] (t);
			\draw[edge, dotted, color1] (s1) to[out=0, in=180] (t);
			\draw[edge, dotted, color1] (s1) to[out=-20, in=200] (t);
			\draw[edge, dotted, color1] (s1) to[out=-40, in=220] (t);
			\draw[edge, dotted, color1] (s1) to[out=-60, in=240] (t);
			
			\coordinate[below of=s2p, yshift=-1.5\nodedistance, xshift=.5\nodedistance]
			(c5);
			\draw[edge, dashed, color1] (s2pv) .. controls (s1) and (c5) .. (t);
			
			\end{tikzpicture}
		\end{minipage}
		\begin{minipage}[t]{.32\textwidth}%
			\centering
			\begin{tikzpicture}[node distance=\nodedistance]
			\clip(-1,1.7) rectangle (3.6, -1);
			\node at (-0.8,1.5) (x) {c)};
			
			\node[vertex] (s1) {};
			\labelLeft{s1}{$s_1$};
			\node[vertex, xshift=3\nodedistance] (t) {};
			\labelRight{t}{$t$};
			
			\path[name path= p1] (s1) to[out=45, in=135] (t); 
			
			\coordinate[right of=s1, yshift=\nodedistance] (c1);
			\coordinate[right of=s1] (c2);
			\path[name path=leftmid] (c1) -- (c2);
			\coordinate[left of=t, yshift=\nodedistance] (c3);
			\coordinate[left of=t] (c4);
			\path[name path=rightmid] (c3) -- (c4);
			\path[name intersections={of=leftmid and p1}];
			\coordinate (s2) at (intersection-1);
			\labelAbove{s2}{$s_2$}
			\path[name intersections={of=rightmid and p1}];
			\coordinate (s2p) at (intersection-1);
			\labelAbove{s2p}{$s_2^\prime$}
			\draw[edge, color1] (s2) to (s2p);
			\node[vertex] at (s2) (s2v) {};
			\node[vertex] at (s2p) (s2pv) {};
			
			\draw[edge, dotted, color1] (s1) to[out=20, in=160] (t);
			\draw[edge, dotted, color1] (s1) to[out=0, in=180] (t);
			\draw[edge, dotted, color1] (s1) to[out=-20, in=200] (t);
			\draw[edge, dotted, color1] (s1) to[out=-40, in=220] (t);
			\draw[edge, dotted, color1] (s1) to[out=-60, in=240] (t);
			
			\coordinate[below of=s2p, yshift=-1.5\nodedistance, xshift=.5\nodedistance]
			(c5);
			\draw[edge, dashed, color1] (s2pv) .. controls (s1) and (c5) .. (t);
			
			\end{tikzpicture}
		\end{minipage}
		\caption{Paths in Case 1 of the proof of Theorem~\ref{the:main}. The dotted
			lines are mutually internally disjoint. The dashed line is
			edge-disjoint to the dotted lines. The solid line is a single edge not
			contained in any of the displayed paths. The colored lines form $k+2$
			edge-disjoint paths of which $k+1$ are internally disjoint.}
		\label{gt::fig:stskeins_s2on}
	\end{figure}
	
	Thus from now on, in addition to~\eqref{gt::thm:bhc_l2_helper_sub1}, we may
	assume:
	\begin{align}
		\label{gt::thm:bhc_l2_helper_sub2}
		\begin{minipage}{.9\textwidth}
			\textit{Let $G^\prime$ be a graph with $\card{E(G^\prime)}=\card{E(G)}$ and
				vertices $s_1^\prime, s_2^\prime, t^\prime\in V(G^\prime)$ such that
				$s_1^\prime\neq t$ and $s_2^\prime$ is contained in an
				\skein{s_1}{t}{(k+1)}.
				If Properties~\ref{prop:1} through~\ref{prop:3} are
				satisfied, then there exist $k+2$ edge-disjoint paths of which $k+1$ are
				internally disjoint $s_1^\prime$-$t^\prime$ paths and of which one is
				an $s_2^\prime$-$t^\prime$ path.}
		\end{minipage} 
	\end{align}
	
	\noindent\textbf{Case 2:} The vertex $s_2$ is not contained in any
	\skein{s_1}{t}{(k+1)}.
	
	Denote by $s_2^\prime$ a vertex on an \skein{s_1}{t}{(k+1)} that is closest
	(with respect to the number of edges) to $s_2$ among all vertices on
	\skeins{s_1}{t}{(k+1)}, \textit{cf.}~Figure~\ref{gt::fig:stskeins_s2gone}. Now we show that the assumptions still hold if we
	replace $s_2$ by $s_2^\prime$.

	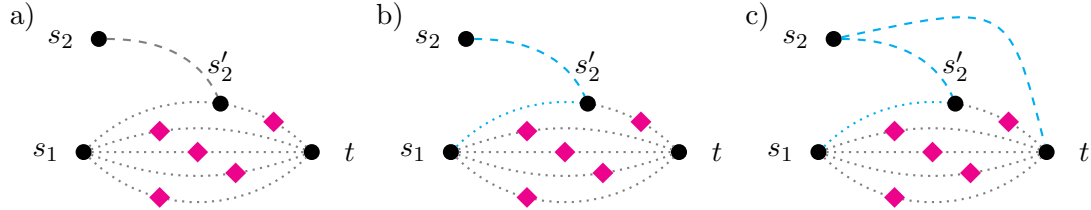
\begin{figure}[ht]
		\begin{minipage}[t]{.32\textwidth}%
			\centering
			\begin{tikzpicture}[node distance=\nodedistance]
			\clip(-1,2.1) rectangle (3.6, -1);
			\node at (-0.8,1.8) (x) {a)};
			\node[vertex] (s1) {};
			\labelLeft{s1}{$s_1$};
			\node[vertex, xshift=3\nodedistance] (t) {};
			\labelRight{t}{$t$};
			\node[vertex, above of=s1, yshift=.5\nodedistance, xshift=.2\nodedistance]
			(s2) {};
			\labelLeft{s2}{$s_2$}
			
			\draw[name path= l5, edge, dotted] (s1) to[out=45, in=135] (t); 
			\path[name path=s1skein] (s2) -- (t);
			\path[name intersections={of=l5 and s1skein}];
			\draw[edge, dashed] (s2) to[out=0, in=110] (intersection-1);
			\node[vertex] (s2p) at (intersection-1) {};
			\labelAbove{s2p}{$s_2^\prime$}
			\draw[name path=l1, edge, dotted] (s1) to[out=20, in=160] (t);
			\draw[name path=l2, edge, dotted] (s1) to (t);
			\draw[name path=l3, edge, dotted] (s1) to[out=-20, in=200] (t);
			\draw[name path=l4, edge, dotted] (s1) to[out=-45, in=225] (t);
			
			\coordinate[right of=s1, xshift=.5\nodedistance, yshift=\nodedistance] (m1);
			\coordinate[right of=s1, xshift=.5\nodedistance, yshift=-\nodedistance]
			(m2);
			\coordinate[right of=s1, yshift=\nodedistance] (c1);
			\coordinate[right of=s1, yshift=-\nodedistance] (c2);
			\coordinate[left of=t, yshift=\nodedistance] (c3);
			\coordinate[left of=t, yshift=-\nodedistance] (c4);
			\path[name path=mid-l1] (c1) -- (c2);
			\path[name path=mid-l2] (m1) -- (m2);
			\path[name path=mid-l3] (c3) -- (c4);
			\path[name path=mid-l4] (c1) -- (c2);
			\coordinate[left of=t, xshift=.5\nodedistance, yshift=\nodedistance] (c5);
			\coordinate[left of=t, xshift=.5\nodedistance, yshift=-\nodedistance] (c6);
			\path[name path=mid-l5] (c5) -- (c6);
			\foreach \x in {l1, l2, l3, l4, l5} {
				\path[name intersections={of={mid-\x} and \x}];
				\node[diamond, fill=color2,inner sep=2pt] at (intersection-1) (sep-\x) {}; 
			}
			\end{tikzpicture}
		\end{minipage}
		\begin{minipage}[t]{.32\textwidth}%
			\centering
			\begin{tikzpicture}[node distance=\nodedistance]
			\clip(-1,2.1) rectangle (3.6, -1);
			\node at (-0.8,1.8) (x) {b)};
			\node[vertex] (s1) {};
			\labelLeft{s1}{$s_1$};
			\node[vertex, xshift=3\nodedistance] (t) {};
			\labelRight{t}{$t$};
			\node[vertex, above of=s1, yshift=.5\nodedistance, xshift=.2\nodedistance]
			(s2) {};
			\labelLeft{s2}{$s_2$}
			
			\path[name path= l5] (s1) to[out=45, in=135] (t); 
			\path[name path=s1skein] (s2) -- (t);
			\path[name intersections={of=l5 and s1skein}];
			\coordinate[above of=intersection-1] (xxx);
			\begin{scope}
			\clip(s1)rectangle(xxx);
			\draw[edge, dotted, color1] (s1) to[out=45, in=135] (t);
			\end{scope}
			\begin{scope}
			\clip(intersection-1)rectangle(t);
			\draw[edge, dotted] (s1) to[out=45, in=135] (t);
			\end{scope}
			\draw[edge, dashed, color1] (s2) to[out=0, in=110] (intersection-1);
			\node[vertex] (s2p) at (intersection-1) {};
			\labelAbove{s2p}{$s_2^\prime$}
			\draw[name path=l1, edge, dotted] (s1) to[out=20, in=160] (t);
			\draw[name path=l2, edge, dotted] (s1) to (t);
			\draw[name path=l3, edge, dotted] (s1) to[out=-20, in=200] (t);
			\draw[name path=l4, edge, dotted] (s1) to[out=-45, in=225] (t);
			
			\coordinate[right of=s1, xshift=.5\nodedistance, yshift=\nodedistance] (m1);
			\coordinate[right of=s1, xshift=.5\nodedistance, yshift=-\nodedistance]
			(m2);
			\coordinate[right of=s1, yshift=\nodedistance] (c1);
			\coordinate[right of=s1, yshift=-\nodedistance] (c2);
			\coordinate[left of=t, yshift=\nodedistance] (c3);
			\coordinate[left of=t, yshift=-\nodedistance] (c4);
			\path[name path=mid-l1] (c1) -- (c2);
			\path[name path=mid-l2] (m1) -- (m2);
			\path[name path=mid-l3] (c3) -- (c4);
			\path[name path=mid-l4] (c1) -- (c2);
			\coordinate[left of=t, xshift=.5\nodedistance, yshift=\nodedistance] (c5);
			\coordinate[left of=t, xshift=.5\nodedistance, yshift=-\nodedistance] (c6);
			\path[name path=mid-l5] (c5) -- (c6);
			\foreach \x in {l1, l2, l3, l4, l5} {
				\path[name intersections={of={mid-\x} and \x}];
				\node[diamond, fill=color2,inner sep=2pt] at (intersection-1) (sep-\x) {}; 
			}
			\end{tikzpicture}
		\end{minipage}
		\begin{minipage}[t]{.32\textwidth}%
			\centering
			\begin{tikzpicture}[node distance=\nodedistance]
			\clip(-1,2.1) rectangle (3.6, -1);
			\node at (-0.8,1.8) (x) {c)};
			\node[vertex] (s1) {};
			\labelLeft{s1}{$s_1$};
			\node[vertex, xshift=3\nodedistance] (t) {};
			\labelRight{t}{$t$};
			\node[vertex, above of=s1, yshift=.5\nodedistance, xshift=.2\nodedistance]
			(s2) {};
			\labelLeft{s2}{$s_2$}
			
			\path[name path= l5] (s1) to[out=45, in=135] (t); 
			\path[name path=s1skein] (s2) -- (t);
			\path[name intersections={of=l5 and s1skein}];
			\coordinate[above of=intersection-1] (xxx);
			\begin{scope}
			\clip(s1)rectangle(xxx);
			\draw[edge, dotted, color1] (s1) to[out=45, in=135] (t);
			\end{scope}
			\begin{scope}
			\clip(intersection-1)rectangle(t);
			\draw[edge, dotted] (s1) to[out=45, in=135] (t);
			\end{scope}
			\draw[edge, dashed, color1] (s2) to[out=0, in=110] (intersection-1);
			\node[vertex] (s2p) at (intersection-1) {};
			\labelAbove{s2p}{$s_2^\prime$}
			\draw[name path=l1, edge, dotted] (s1) to[out=20, in=160] (t);
			\draw[name path=l2, edge, dotted] (s1) to (t);
			\draw[name path=l3, edge, dotted] (s1) to[out=-20, in=200] (t);
			\draw[name path=l4, edge, dotted] (s1) to[out=-45, in=225] (t);
			
			\coordinate[right of=s1, xshift=.5\nodedistance, yshift=\nodedistance] (m1);
			\coordinate[right of=s1, xshift=.5\nodedistance, yshift=-\nodedistance]
			(m2);
			\coordinate[right of=s1, yshift=\nodedistance] (c1);
			\coordinate[right of=s1, yshift=-\nodedistance] (c2);
			\coordinate[left of=t, yshift=\nodedistance] (c3);
			\coordinate[left of=t, yshift=-\nodedistance] (c4);
			\path[name path=mid-l1] (c1) -- (c2);
			\path[name path=mid-l2] (m1) -- (m2);
			\path[name path=mid-l3] (c3) -- (c4);
			\path[name path=mid-l4] (c1) -- (c2);
			\coordinate[left of=t, xshift=.5\nodedistance, yshift=\nodedistance] (c5);
			\coordinate[left of=t, xshift=.5\nodedistance, yshift=-\nodedistance] (c6);
			\path[name path=mid-l5] (c5) -- (c6);
			\foreach \x in {l1, l2, l3, l4, l5} {
				\path[name intersections={of={mid-\x} and \x}];
				\node[diamond, fill=color2,inner sep=2pt] at (intersection-1) (sep-\x) {}; 
			}
			\coordinate[above of=c5] (cstar);
			\draw[edge, dashed, color1] (s2) .. controls (cstar) .. (t);
			\end{tikzpicture}
		\end{minipage}
		\caption{Case 2 in the proof of
			Theorem~\ref{the:main}: Supposed separation of
			$\{s_1,s_2^\prime\}$ and $t$. The colored diamonds correspond to
			$(W,F)$. The dotted lines are mutually internally disjoint. The
			colored
			lines indicate a connection of vertices that does not touch $(W,F)$.}
		\label{gt::fig:stsep_s2gone}
	\end{figure}	
	
	Observe that Properties~\ref{prop:1} and \ref{prop:2} are
	satisfied when replacing $s_2$ by $s_2^\prime$. To see that
	Property~\ref{prop:3} still holds, suppose that there exists an
	$\{s_1,s_2^\prime\}$-$t$ disconnecting pair $(W, F)$ of cardinality $k+1$ and
	order at most $k$. As there cannot be any $s_1$-$t$ path left in $G-W-F$, all
	elements of the disconnecting pair are contained in some \skein{s_1}{t}{(k+1)},
	\textit{cf.}~Figure~\ref{gt::fig:stsep_s2gone} a). The vertex set $W$ may also not
	contain $s_2^\prime$ by definition. Thus, the vertices $s_1$, $s_2$ and
	$s_2^\prime$ are contained in the same component of $G-W-F$,
	\textit{cf.}~Figure~\ref{gt::fig:stsep_s2gone} b). In $G$, the pair $(W,F)$ cannot be
	$\{s_1,s_2\}$-$t$ disconnecting by assumption,
	\textit{cf.}~Figure~\ref{gt::fig:stsep_s2gone} c). This contradicts $(W,F)$ being
	$\{s_1,s_2^\prime\}$-$t$ disconnecting in $G$ and
	Property~\ref{prop:3} is satisfied.

	\begin{figure}[ht]
		\begin{minipage}[t]{.32\textwidth}%
			\centering
			\begin{tikzpicture}[node distance=\nodedistance]
				\clip(-1,2.1) rectangle (3.6, -1);
				\node at (-0.8,1.8) (x) {a)};
				\node[vertex] (s) {};
				\labelLeft{s}{$s_1$};
				\node[vertex, xshift=3\nodedistance] (t) {};
				\labelRight{t}{$t$};
				\node[vertex, above of=s, yshift=.5\nodedistance, xshift=.2\nodedistance]
				(s1) {};
				\labelLeft{s1}{$s_2$}
				
				\draw[name path= p1, edge, dotted] (s) to[out=45, in=135] (t); 
				\path[name path=s1skein] (s1) -- (t);
				\path[name intersections={of=p1 and s1skein}];
				\draw[edge, dashed] (s1) to[out=0, in=110] (intersection-1);
				\node[vertex] (s2) at (intersection-1) {};
				\labelAbove{s2}{$s_2^\prime$}
				\draw[edge, dotted] (s) to[out=20, in=160] (t);
				\draw[edge, dotted] (s) to (t);
				\draw[edge, dotted] (s) to[out=-20, in=200] (t);
				\draw[edge, dotted] (s) to[out=-45, in=225] (t);
			\end{tikzpicture}
		\end{minipage}
		\begin{minipage}[t]{.32\textwidth}%
			\centering
			\begin{tikzpicture}[node distance=\nodedistance]
				\clip(-1,2.1) rectangle (3.6, -1);
				\node at (-0.8,1.8) (x) {b)};
				\node[vertex] (s) {};
				\labelLeft{s}{$s_1$};
				\node[vertex, xshift=3\nodedistance] (t) {};
				\labelRight{t}{$t$};
				\node[vertex, above of=s, yshift=.5\nodedistance, xshift=.2\nodedistance]
				(s1) {};
				\labelLeft{s1}{$s_2$}
				
				\path[name path= p1] (s) to[out=45, in=135] (t); 
				\path[name path=s1skein] (s1) -- (t);
				\path[name intersections={of=p1 and s1skein}];
				\draw[edge, dashed] (s1) to[out=0, in=110] (intersection-1);
				\node[vertex] (s2) at (intersection-1) {};
				\labelAbove{s2}{$s_2^\prime$}
				
				\draw[edge, dotted, color1] (s) to[out=10, in=170] (t);
				\draw[edge, dotted, color1] (s) to[out=-10, in=190] (t);
				\draw[edge, dotted, color1] (s) to[out=-30, in=210] (t);
				\draw[edge, dotted, color1] (s) to[out=-50, in=230] (t);
				
				\coordinate[below of=s1,yshift=.5\nodedistance, xshift=.8\nodedistance]
				(c4);
				\coordinate[left of=t, yshift=.3\nodedistance, xshift=0\nodedistance] (c3);
				\coordinate[above of=t, yshift=-.5\nodedistance] (c5);
				\draw[edge, dotted, color1] (s) to[out=90, in=180] (c4);
				\draw[edge, dotted, color1] (c4) to[out=0, in=180] (c3);
				\draw[edge, dotted, color1] (c3) to[out=0, in=180] (c5);
				\draw[edge, dotted, color1] (c5) to[out=0, in=45] (t);
				
				\coordinate[below of= s1, xshift=-.5\nodedistance] (c2);
				\coordinate[right of =s1, yshift=2\nodedistance] (c);
				\draw[edge, dashed, color1] (s2) .. controls (c2) and (c) .. (t);
			\end{tikzpicture}
		\end{minipage}
		\begin{minipage}[t]{.32\textwidth}%
			\centering
			\begin{tikzpicture}[node distance=\nodedistance]
				\clip(-1,2.1) rectangle (3.6, -1);
				\node at (-0.8,1.8) (x) {c)};
				\node[vertex] (s) {};
				\labelLeft{s}{$s_1$};
				\node[vertex, xshift=3\nodedistance] (t) {};
				\labelRight{t}{$t$};
				\node[vertex, above of=s, yshift=.5\nodedistance, xshift=.2\nodedistance]
				(s1) {};
				\labelLeft{s1}{$s_2$}
				
				\path[name path= p1] (s) to[out=45, in=135] (t); 
				\path[name path=s1skein] (s1) -- (t);
				\path[name intersections={of=p1 and s1skein}];
				\draw[name path=goal, edge, dashed] (s1) to[out=0, in=110] (intersection-1);
				\node[vertex] (s2) at (intersection-1) {};
				\labelAbove{s2}{$s_2^\prime$}
				
				\draw[edge, dotted, color1] (s) to[out=10, in=170] (t);
				\draw[edge, dotted, color1] (s) to[out=-10, in=190] (t);
				\draw[edge, dotted, color1] (s) to[out=-30, in=210] (t);
				\draw[edge, dotted, color1] (s) to[out=-50, in=230] (t);
				
				\coordinate[below of=s1,yshift=.5\nodedistance, xshift=.8\nodedistance]
				(c4);
				\coordinate[left of=t, yshift=.3\nodedistance, xshift=0\nodedistance] (c3);
				\coordinate[above of=t, yshift=-.5\nodedistance] (c5);
				\draw[edge, dotted, color1] (s) to[out=90, in=180] (c4);
				\draw[edge, dotted, color1] (c4) to[out=0, in=180] (c3);
				\draw[edge, dotted, color1] (c3) to[out=0, in=180] (c5);
				\draw[edge, dotted, color1] (c5) to[out=0, in=45] (t);
				
				\coordinate[below of= s1, xshift=-.5\nodedistance] (c2);
				\coordinate[right of =s1, yshift=2\nodedistance] (c);
				
				\path[name path= final, edge, dashed] (s2) .. controls (c2) and (c) .. (t);
				\path[name intersections={of=final and goal}];
				
				\coordinate[right of=intersection-1, xshift=\nodedistance,
				yshift=.5\nodedistance] (c9);
				\coordinate[left of=intersection-1] (c8);
				\coordinate[below of=s2] (c7);
				\begin{scope}
					\clip(t)rectangle(s2)(c9)rectangle(s2)(intersection-1)rectangle(c9);
					\draw[edge, dashed, color=color1] (s2) .. controls (c2) and (c) .. (t);
				\end{scope}
				\begin{scope}
					\clip(c7)rectangle(c8);
					\draw[edge, dashed] (s2) .. controls (c2) and (c) .. (t);
				\end{scope}
				\node[vertex] (sp) at (intersection-1) {};
				
				\draw[edge, dashed, color=color1] (s1) to[out=0, in=170] (sp);
				\labelAbove{sp}{$s^\prime$}
			\end{tikzpicture}
		\end{minipage}
		\caption{Paths in Case 2 of the proof of Theorem~\ref{the:main}. The dotted
	lines are internally disjoint. The dashed lines are edge-disjoint to the
	dotted lines. The colored lines form $k+2$ edge-disjoint paths of which
	$k+1$ are internally disjoint.} \label{gt::fig:stskeins_s2gone}
	\end{figure}
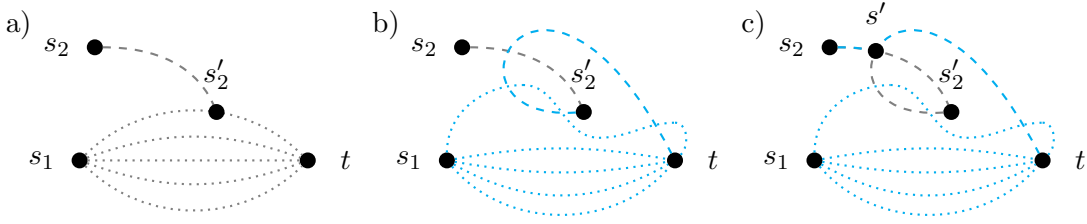
	
	Thus, by~\eqref{gt::thm:bhc_l2_helper_sub2} there exist $k+2$ edge-disjoint
	paths, say $P_1,\dots,P_{k+2}$, such that the paths $P_1,\dots P_{k+1}$ are internally
	disjoint $s_1$-$t$ paths and $P_{k+2}$ is an $s_2^\prime$-$t$ path,
	\textit{cf.}~Figure~\ref{gt::fig:stskeins_s2gone}~b). Denote by $P^\prime$ a shortest
	$s_2$-$s_2^\prime$ path. Note that no element of $P^\prime$, except possibly
	$s_2^\prime$, is contained in $P_1,\dots,P_{k+1}$ as no vertex or edge on an
	\skein{s_1}{t}{(k+1)} is closer to $s_2$ than $s_2^\prime$. Further denote by
	$s^\prime$ the vertex on $P^\prime$ closest to $s_2$ that is also contained in
	$P_{k+2}$, \textit{cf.}~Figure~\ref{gt::fig:stskeins_s2gone} c). Then $P^\prime s^\prime
	\cup s^\prime P_{k+2}$ is an $s_2$-$t$ path that is edge-disjoint to all
	$P_1,\dots, P_{k+1}$ and we obtain the desired paths.
\end{proof}

\begin{proof}[Proof of Theorem~\ref{the:main_result}]
	If $s$ and $t$ are adjacent, then $(k,1)$ is a connectivity pair in $G-st$ and
there exist $k+1$ internally disjoint $s$-$t$ paths in $G$ by
Observation~\ref{obs:bineke_harary_l=1_or_k=0}~\ref{obs:l=1}. Together with the
deleted edge we get the desired paths in $G$. So assume that $s$ and $t$ are
not adjacent. We show that Properties~\ref{prop:1} through~\ref{prop:3} of
Theorem~\ref{the:main} hold for $G$, $s_1=s_2=s$, and $t$.
	
	By the definition of a connectivity pair, there is no $s$-$t$ disconnecting
pair of cardinality less than $k+2$, order at most $k$ and size at most $2$.
Thus, by Menger's Theorem there exist $k+1$ internally disjoint $s$-$t$ paths
and Properties~\ref{prop:1} and~\ref{prop:2} are satisfied. Now suppose
Property~\ref{prop:3} is not satisfied and let $(W,F)$ be an $s$-$t$
disconnecting pair of cardinality $k+1$ and order at most $k$. As $s$ and $t$
are not adjacent, any edge in $F$ has an endvertex not contained in~$\{s,t\}$.
Thus, replacing all but one edge in $(W,F)$ with one of its endvertices that is
not contained in~$\{s,t\}$ we get an $s$-$t$ disconnecting pair of cardinality
$k+1$, order $k$, and size $1$. Such a pair does not exist, as $(k,2)$ is a
connectivity pair for $s$ and $t$ yielding a contradiction. Thus, the
assumptions of Theorem~\ref{the:main} are satisfied and there exist $k+2$
edge-disjoint $s$-$t$ paths in~$G$ of which $k+1$ are internally disjoint.
\end{proof}

Theorem~\ref{the:main_result} is not only a stand-alone result, but can also be
of help when proving the Beineke-Harary-Conjecture for some restricted graph
classes. We illustrate this fact by proving the conjecture for graphs with
treewidth at most~$3$ in the next section.

\section[The Beineke-Harary-Conjecture for Graphs with Small Treewidth]{The
	Beineke-Harary-Conjecture for Graphs with Small Treewidth}
\label{sec:treewidth}
In this section we prove the Beineke-Harary-Conjecture for graphs of treewidth
at most $3$. To this end we recall the definition of treewidth and some 
basic results on tree decompositions. For more details, see~\cite{Bodlaender98}.

For a graph $G$ a \emph{tree decomposition} $(\mathcal{B},\mathcal{T})$ of $G$ 
consists of a tree~$\mathcal{T}$ and a set~$\mathcal{B} = \{B_i \colon i \in 
V(\mathcal{T}) \}$
of \emph{bags} $B_i \subseteq V(G)$ such that
$V(G) = \bigcup_{i \in V(\mathcal{T})} B_i$.
Further, for each edge~${vw \in E(G)}$ there exists a node $i \in V(\mathcal{T})$
such 
that
$v,w \in B_i$, and
if $v \in B_{j_1} \cap B_{j_2},$ then~${v \in B_i}$ for each node
$i$  on the simple path connecting $j_1$ and $j_2$ in $\mathcal{T}$.
A tree decomposition $(\mathcal{B},\mathcal{T})$ has \emph{width} $k$ if each 
bag is of cardinality at most $k+1$ and there
exists some bag of size $k+1$.
The \emph{treewidth} of $G$ is the smallest integer $k$ for which there is a 
width $k$ tree decomposition of $G$.
We write $\tw(G) = k$.
We call a tree decomposition \emph{small} if no bag is completely contained in
any other bag. Every graph $G$ has a small tree decomposition of width $\tw(G)$.

The following result is well known and can, for example, be found in \cite{Diestel00}. We
formulate it here as an observation:
\begin{observation}\label{obs:tw_separator}
	Let $G$ be a graph and $(\mathcal{B},\mathcal{T})$ a tree decomposition of 
	$G$.
	Let $ij\in E(\mathcal{T})$ and denote by $T_i$ and $T_j$ the two 
	subtrees of $\mathcal{T}-ij$ with $i\in V(T_i)$ and $j\in V(T_j)$. If $u\in 
	B_{i^\prime}\setminus (B_i\cap B_j)$ for some $i^\prime\in V(T_i)$ and 
	$v\in B_{j^\prime}\setminus(B_i\cap B_j)$ for some $j^\prime\in V(T_j)$, 
	then $B_i\cap B_j$ is a separator for $u$ and $v$ in $G$. 
\end{observation}

We are now ready to prove the Beineke-Harary-Conjecture for a subclass of graphs
of treewidth at most $3$. 

\begin{lemma}
	\label{lem:bh_tw3_st_separated} Let $G$ be a graph of treewidth at most
	$3$, let $s,t\in V(G)$ be distinct and non-adjacent, and let $k\geq 0$ and
	$l\geq 1$ be integers. Further, let $(\mathcal{B},\mathcal{T})$ be a tree
	decomposition of width at most $3$ such that for no bag $B\in\mathcal{B}$ we
	have $s,t\in B$. If $(k,l)$ is a connectivity pair for $s$ and $t$ in $G$, then
	there exist $k+l$ edge-disjoint $s$-$t$ paths of which $k+1$ are internally
	disjoint.
\end{lemma}
\begin{proof}
	Denote by $T_s$ ($T_t$) the subtree of $\mathcal{T}$ induced by all nodes
	corresponding to bags containing $s$ ($t$). As $V(T_s)\cap V(T_t)=\emptyset$,
	there exists an edge $ij\in E(\mathcal{T})$ that separates $V(T_s)$ from
	$V(T_t)$. Thus, by Observation~\ref{obs:tw_separator}, the set $B_i\cap
	B_j$ is an $s$-$t$ vertex-separator in~$G$. We may assume without loss of
	generality that
	$B_i\neq B_j$ and therefore get ${\card{B_i\cap B_j}\leq 3}$. The pair $(k,l)$ is
	a connectivity pair and $l\geq 1$, which implies $k\leq 2$. If $l=1$ the result
	follows from Observation~\ref{obs:bineke_harary_l=1_or_k=0}~\ref{obs:l=1}. If
	$l=2$ the
	result follows from Theorem~\ref{the:main_result}. Further, if $k=1$ the result
	follows form Corollary~\ref{cor::enomoto_kaneko_k=1}. Finally if $l > 2$,
	$k=2$, and $q,r$ integers such that $2+l=q\cdot 3 +r$ with $1\leq r\leq 3$, we
	get that $q+r>2=k$ and the desired result follows from
	Theorem~\ref{thm:enomoto_kaneko}.
\end{proof}

Note that in the proof of Lemma~\ref{lem:bh_tw3_st_separated}, we used our
main result, Theorem~\ref{the:main_result}, from the previous section. It is
worth noting, that no other result in this paper or another \emph{easy}
argument seems to be able to replace Theorem~\ref{the:main_result} in the proof
of the lemma. In fact, the proof of the theorem was the only piece missing for
proving the Beineke-Harary-Conjecture for graphs of treewidth at most $3$ for
some time.

We observe that the Beineke-Harary-Conjecture holds for graphs of treewidth~$1$:
If for a graph~$G$ the underlying simple graph is a tree, either $s$ and $t$ are
adjacent or there exists a vertex~$a\in V(G)\setminus\{s,t\}$ separating $s$ and
$t$. In both cases the only possible connectivity pairs for $s$ and $t$ are of
the form $(0,l)$ for $l\geq 1$ or $(1,0)$. The conjecture follows from
Observation~\ref{obs:bineke_harary_l=1_or_k=0}.

\begin{observation}\label{obs:tw_1}
	Let $G$ be a graph of treewidth $1$ and vertices $s,t\in V(G)$. Further, let 
	$(k,l)$ be a connectivity pair for $s$ and $t$ in $G$, with $k\geq 0$ and 
	$l\geq 1$. Then there exist $k+l$ edge-disjoint $s$-$t$ paths, $k+1$ of 
	which are internally disjoint.
\end{observation}

In the next step we prove the conjecture for graphs of treewidth at
most~$2$. Although Theorem~\ref{thm:conjecture_on_tw_2} is implied by
Theorem~\ref{thm:conjecture_tw_3} and the proof could be included into the one
of Theorem~\ref{thm:conjecture_tw_3}, for better readability, we prove the
theorems separately. The structure of the two proofs is similar and therefore
the proof of Theorem~\ref{thm:conjecture_on_tw_2} can be regarded as a warm-up
for the one of Theorem~\ref{thm:conjecture_tw_3}. The following lemma comes in
handy to establish a case distinction in the proofs of the Beineke-Harary-Conjecture for graphs with small treewidth. It allows us to use
Lemma~\ref{lem:bh_tw3_st_separated} in the main proof.

\begin{lemma}\label{lem:tree_decomps_two_bags}
	Let $G$ be a graph of treewidth at most $k$ for some integer $k\geq 1$ and
	let $s,t\in V(G)$ be distinct and non-adjacent. Assume that every tree
	decomposition of width at most $k$ has a bag containing $s$ and $t$. Then,
	there exists a tree decomposition $D=(\mathcal{B},\mathcal{T})$ of width $k$,
	such that there is some $ij\in E(\mathcal{T})$ with $s,t\in B_i\cap B_j$,
	$\card{B_i\cap B_j}\leq k$ and $G-(B_i\cap B_j)$ not connected. In particular,
	$s$ and $t$ are contained in a vertex-separator in $G$ containing at most $k$
	vertices.
\end{lemma}
\begin{proof}
	Let $G$ be a graph of treewidth $k$ and $s,t\in V(G)$ distinct and
	non-adjacent.
	Moreover, assume that every tree decomposition of width $k$ has a bag that 
	contains $s$ and~$t$.
	We claim that in this case every tree decomposition
	of width $k$ has at least two bags containing $s$ and $t$.
	
	Suppose that $\td$ is a tree decomposition of $G$ of treewidth $k$, such 
	that $s$ and $t$ share exactly one bag $B_i$. Let $j_1,\dots, j_r$ be the 
	neighbors of 
	$i$ in $\mathcal{T}$ whose bags contain $s$. We construct a new tree 
	decomposition by replacing the node~$i$ with two adjacent nodes $i_1$ and 
	$i_2$ with corresponding bags $B_{i_1}=B_i\setminus\{t\}$ and 
	$B_{i_2}=B_i\setminus\{s\}$, making $j_1,\dots, j_r$ adjacent to $i_1$ and 
	making the remaining neighbors of $i$ adjacent to $i_2$. As $s$ and $t$ are 
	not adjacent, the result is in fact a tree decomposition of width at most 
	$k$, in which no bag contains both $s$ and~$t$. This is a contradiction. Thus,
	we may assume that every tree decomposition of width at
	most $k$ has at least two bags containing $s$ and $t$.
	
	Consider a small tree decomposition $(\mathcal{B},\mathcal{T})$ of width $k$.
	By the arguments above we may assume that there is an edge $ij\in E(\mathcal{T})$ such that $s,t\in B_i\cap B_j$.
	Since $(\mathcal{B},\mathcal{T})$ is a small tree decomposition we have $B_i\nsubseteq B_j$ and $B_j\nsubseteq B_i$.
	Thus, there is a $u\in B_i\setminus B_j$ and a $v\in B_j\setminus B_i$ and by Observation~\ref{obs:tw_separator} the set $W=B_i\cap B_j$ separates $u$ and $v$.
	Since $W$ has at most $k$ vertices, $s,t\in W$ and $G-W$ is disconnected, the lemma follows.
\end{proof}

\begin{theorem}\label{thm:conjecture_on_tw_2}
	Let $G$ be a graph of treewidth at most $2$ with distinct vertices $s,t\in 
	V(G)$ and $k\geq 0$ and $l\geq 1$ integers.
	If $(k,l)$ is a connectivity pair for $s$ and $t$, then $G$ contains $k+l$ edge-disjoint $s$-$t$ paths of which $k+1$ are internally disjoint.
\end{theorem}
\begin{proof}
	Let $G$ be a graph, $s,t\in V(G)$ be distinct vertices and let $(k,l)$ be a
	connectivity pair for $s$ and $t$ with $l\geq1$. If $\tw(G)=1$ the result
	follows from Observation~\ref{obs:tw_1}. By
	Lemma~\ref{lem:conjecture_non_adjacent_suffices} we may assume that $s$ and	$t$ are not adjacent in~$G$.

	We prove the theorem by induction on the number of vertices $\card{V(G)}$. If
$|V(G)|\leq 3$, as $s$ and $t$ are not adjacent, there always exists a tree
decomposition of $G$ in which no bag contains both, $s$ and $t$. The claim
follows from Lemma~\ref{lem:bh_tw3_st_separated}. So assume the claim holds for
all graphs with less than $|V(G)|$ vertices.
	
If there exists a tree decomposition of $G$ in which no bag contains both $s$
and $t$, the claim is again implied by Lemma~\ref{lem:bh_tw3_st_separated}.
Otherwise, by Lemma~\ref{lem:tree_decomps_two_bags}, the set $\{s,t\}$ is a
vertex-separator in~$G$. Let $C$ be a component of $G-\{s,t\}$ and denote the
graph induced by $C\cup\{s,t\}$ by $G_1$. Let $G_2=G - C$. Note that
$|V(G_i)|<|V(G)|$ for $i\in\{1,2\}$ and $E(G_1)\cap E(G_2)=\emptyset$. Consider
some $s$-$t$ disconnecting pair $(W,F)$ of order~$k$ and size~$l$ in $G$. For
$i\in\{1,2\}$, the pair induces an $s$-$t$ disconnecting pair $(W_i,F_i)$ in
$G_i$, with $W_i=W\cap V(G_i)$ and $F_i=F\cap E(G_i)$. Let $k_i=\card{W_i}$ and
$l_i=\card{F_i}$. Then, $(k_i,l_i)$ is a connectivity pair for $s$ and $t$ in
$G_i$. Further, $k_1+k_2=k$, $l_1+l_2=l$ and without loss of generality we may
assume $l_2\geq 1$ since $E_G(s,t)\subseteq E(G_2)$. Thus, in $G_1$ there exist $k_1+l_1$ edge-disjoint paths of
which $k_1$ are internally disjoint. Note that we cannot assume that there
exist $k_1+1$ internally disjoint paths as $l_1$ may equal $0$. In $G_2$, by
induction, we get $k_2+l_2$ edge-disjoint paths of which $k_2+1$ are internally
disjoint. For any two paths $P_1$ in $G_1$ and $P_2$ in $G_2$ it holds
that $P_1$ and $P_2$ are internally disjoint in $G$. Thus, there exist
$k_1+k_2+l_1+l_2=k+l$ edge-disjoint $s$-$t$ paths in $G$ of which
$k_1+k_2+1=k+1$ are internally disjoint.
\end{proof}

Finally, we turn to the proof of the Beineke-Harary-Conjecture for graphs of
treewidth at most $3$. The structure of the proof is very similar to the one in
Theorem~\ref{thm:conjecture_on_tw_2}. It is quite possible that this
structure also generalizes to graphs with larger treewidth. The main reason why we
do not prove the conjecture for graphs of treewidth at most $4$ (or even larger)
is, that in order for Lemma~\ref{lem:bh_tw3_st_separated} to hold for this
class of graphs, we would have to prove the conjecture for $l=3$ or find another
way of proving this. The idea of the proof is to divide the graph at some
separator containing $s$ and $t$ and use paths found in the resulting graphs by
induction. In contrast to Theorem~\ref{thm:conjecture_on_tw_2}, a separator
containing $s$ and $t$ may now also contain a third vertex $a$. Thus, it is
possible that some of the searched path actually cross at this vertex. To
address this issue we introduce artificial edges that simulate part of the paths
in the other component.

\begin{theorem}
	\label{thm:conjecture_tw_3} Let $G$ be a graph of treewidth  at most $3$.
	Let $s,t\in V(G)$ be two distinct vertices and $k\geq 0$, $l\geq 1$ integers.
	If $(k,l)$ is a connectivity pair for $s$ and $t$, then $G$ contains $k+l$
	edge-disjoint $s$-$t$ paths of which $k+1$ are internally disjoint.
\end{theorem}
\begin{proof}
	Let $G$ be a graph, $s,t\in V(G)$ distinct vertices and let $(k,l)$ be a
	connectivity pair for $s$ and $t$ with $l\geq 1$. If $\tw(G)\leq 2$ the result
	follows from Theorem~\ref{thm:conjecture_on_tw_2}. By
	Lemma~\ref{lem:conjecture_non_adjacent_suffices} we may assume that $s$ and
	$t$ are not adjacent in~$G$.
	
	As in the proof of Theorem~\ref{thm:conjecture_on_tw_2}, we do the proof by
	induction on the number of vertices. If $\card{V(G)}\leq 4$, there exists a tree
	decomposition of $G$ of width $3$ in which no bag contains both, $s$ and $t$,
	and the claim is implied by Lemma~\ref{lem:bh_tw3_st_separated}. So assume
	the claim holds for all graphs of treewidth at most $3$ and $\card{V(G)}\leq 4$.
	
	If there exists a tree decomposition of $G$ of width $3$ in which no
	bag contains both, $s$ and $t$, the claim is again implied by
	Lemma~\ref{lem:bh_tw3_st_separated}. Otherwise, by
	Lemma~\ref{lem:tree_decomps_two_bags}, there exists a tree decomposition $D=(\mathcal{B},\mathcal{T})$ containing an edge $xy\in\mathcal{T}$ such that for $B\coloneqq B_x\cap B_y$ it holds that $s,t\in B$, $\card{B}\leq 3$, and $G-B$ is not connected. If $B=\set{s,t}$ is a vertex-separator, we may simply
	repeat the arguments in the proof of Theorem~\ref{thm:conjecture_on_tw_2}.
	So assume there is a vertex~$a\in V(G)\setminus\set{s,t}$, such that $B=\set{s,t,a}$.
	Let $C$ be a component of $G-B$, denote by $G_1$ the graph induced by $V(C)\cup
	B$ and let $G_2$ be the graph $G-V(C)-E_G(s,a)-E_G(a,t)$,
	\textit{cf.}~Figure~\ref{fig:cp_construction_tw3}.
	
	\begin{figure}[ht]
		\centering
		\begin{minipage}[t]{.485\textwidth}%
			\centering
			\begin{tikzpicture}[node distance=\nodedistance]
			\clip(-.75\nodedistance,2.5\nodedistance) rectangle(5.75\nodedistance, -2.75\nodedistance);
			\node[vertex] (s) {};
			\labelLeft{s}{$s$}
			\node[vertex, right of=s, xshift=4\nodedistance] (t) {};
			\labelRight{t}{$t$}
			\node[vertex, right of=s, xshift=1.5\nodedistance, color5, fill=color5] (a) {};
			\node[above of=a, yshift=-.6\nodedistance] {$a$};

			\node[above of=s, yshift=.7\nodedistance, xshift=.5\nodedistance] (1a) {};
			\node[right of=1a, xshift=.5\nodedistance] (1b) {};
			\node[above of=t, yshift=.7\nodedistance, xshift=-.5\nodedistance] (1c) {};
			\begin{pgfonlayer}{background}
			\fill[brown!20] \convexpath{s,1a, 1b, 1c, t, a}{.3\nodedistance};
			\end{pgfonlayer}
			
			\node[below of=s, yshift=.3\nodedistance] (2a) {};
			\node[below of=s, yshift=.3\nodedistance, xshift=.3\nodedistance] (2aa) {};
			\node[below of=a, yshift=.3\nodedistance, xshift=-.3\nodedistance] (2ab) {};
			\node[below of=a, yshift=.3\nodedistance, xshift=.3\nodedistance] (2bb) {};
			\node[below of=t, yshift=.3\nodedistance, xshift=-.3\nodedistance] (2bc) {};
			\node[below of=a, yshift=.3\nodedistance] (2b) {};
			\node[below of=t, yshift=.3\nodedistance] (2c) {};
			\node[below of=2a, yshift=-.3\nodedistance, xshift=.5\nodedistance] (3a) {};
			\node[below of=2c, yshift=-.3\nodedistance, xshift=-.5\nodedistance] (3b) {};
			\begin{pgfonlayer}{background}
			\fill[black!20] \convexpath{s,2aa, 2ab, a, 2bb, 2bc, t, 2c, 3b, 3a, 2a, s}{.3\nodedistance};
			\end{pgfonlayer}
			
			\node[vertex, above of=s, xshift=.4\nodedistance] (sx1) {};
			\node[vertex, above of=s, yshift=-.4\nodedistance, xshift=.8\nodedistance] (sx2) {};
			\node[vertex, above of=t, xshift=-.4\nodedistance] (tx1) {};
			\node[vertex, above of=t, yshift=-.4\nodedistance, xshift=-.8\nodedistance] (tx2) {};
			\node[above of=sx1, xshift=.1\nodedistance, yshift=-.3\nodedistance] (y1) {$G_1$};
			\node[vertex, above of=sx1, xshift=.4\nodedistance, yshift=-.7\nodedistance] (ry1) {};
			\node[vertex, right of=y1] (y2) {};
			\node[vertex, right of=sx2, xshift=-.2\nodedistance, yshift=-.2\nodedistance] (y3) {};
			\node[vertex, above of=a, xshift=-.5\nodedistance, yshift=-.2\nodedistance] (z1) {};
			\node[vertex, above of=a, xshift=.5\nodedistance] (z2) {};
			\node[vertex, right of=sx1, yshift=.2\nodedistance] (y4) {};
			\node[vertex, right of=y2, yshift=.1\nodedistance, color5, fill=color5] (s1) {};
			\node[vertex, right of=y2, yshift=-.3\nodedistance, xshift=-.2\nodedistance, color5, fill=color5] (s2) {};
			\node[vertex, below of=tx2, xshift=-.4\nodedistance, yshift=.7\nodedistance] (q1) {};
			\node[vertex, above of=q1, xshift=-.2\nodedistance, yshift=-.5\nodedistance] (q2) {};
			\node[vertex, above of=q2, xshift=.6\nodedistance, yshift=-.2\nodedistance] (q3) {};
			\node[vertex, left of =q3, xshift=.4\nodedistance, yshift=.2\nodedistance] (q4) {};
			\node[vertex, below of=q4, xshift=.2\nodedistance, yshift=.5\nodedistance] (q5) {};
			
			\draw[edge, color5] (z1) -- (z2);
			\draw[edge, dotted] (sx2)--(s2)(t) -- (tx1) (t) -- (tx2)(a)--(t) (s) -- (sx1) (s) -- (sx2)(sx1) -- (sx2) (sx2) -- (z1) (sx2) -- (y3) (y3) -- (z1) (sx1)--(ry1)(ry1)--(y4)(y4)--(y2) (s2)--(y4) (y2) --(s1) (y2) -- (s2) (y3)--(a) (s) -- (y3) (t)--(q1) (q1)--(a) (q1)--(q2) (q2)--(z2) (s2)--(z2) (s2)--(z1) (q2)--(q5) (q5)--(q4) (q5)--(q3) (q5)--(s2) (q5)--(s1) (s1)--(q4) (tx1) --(q3) (tx1) --(q5) (tx2)--(q1) (tx2)--(q5) (q3)--(q4)(a)--(q2);
			
			\node[vertex, below of=s, yshift=-.3\nodedistance, xshift=.6\nodedistance] (xsx1) {};
			\node[vertex, below of=s, yshift=.2\nodedistance, xshift=.8\nodedistance] (xsx2) {};
			\node[vertex, below of=t, xshift=-.2\nodedistance] (xtx1) {};
			\node[vertex, below of=t, yshift=.3\nodedistance, xshift=-.7\nodedistance] (xtx2) {};
			\node[vertex,below of=s, xshift=-0\nodedistance] (xstar) {};
			\node[vertex, below of=xsx1, xshift=-.1\nodedistance, yshift=.4\nodedistance] (xy1) {};
			\node[vertex, right of=xy1] (xy2) {};
			\node[vertex, right of=xsx2, xshift=-.2\nodedistance, yshift=-.2\nodedistance] (xy3) {};
			\node[vertex, below of=a, yshift=.3\nodedistance, xshift=-.4\nodedistance] (xz1) {};
			\node[vertex, below of=a, yshift=.3\nodedistance, xshift=.4\nodedistance, color5, fill=color5] (xz2) {};
			\node[vertex, right of=xsx1, yshift=-.2\nodedistance] (xy4) {};
			\node[vertex, right of=xy2, yshift=.1\nodedistance, color5, fill=color5] (xs1) {};
			\node[vertex, right of=xy2, yshift=.6\nodedistance, xshift=-.4\nodedistance] (xs2) {};
			\node[vertex, above of=xtx2, xshift=-.5\nodedistance, yshift=-.9\nodedistance] (p1) {};
			\node[vertex, below of=p1, xshift=-.2\nodedistance, yshift=.4\nodedistance] (p2) {};
			\node[below of=p2, xshift=.8\nodedistance, yshift=.2\nodedistance] (p3) {$G_2$};
			\node[vertex, left of =p3, xshift=.0\nodedistance, yshift=0\nodedistance] (p4) {};
			\node[vertex, above of=p4, xshift=.8\nodedistance, yshift=-.4\nodedistance] (p5) {};
			
			\draw[edge, color5] (xs2)--(p2);
			\draw[edge, dotted] (xz1)--(xz2)(xsx1)--(xy4)(xs2)--(xy3)(xs1)--(p4) (t)--(p1)(xstar)--(s)(xstar)--(xy1)(s)--(xz1) (s)--(xsx1)(s)--(xsx2)(xsx1)--(xy1)(xsx2)--(xz1)(xz1)--(a)(xz2)--(a)(xy1)--(xy2)(xy2)--(xy4)(xy2)--(xs1)(xy2)--(xs2)(xy4)--(xs2)(xs2)--(xz2)(xs2)--(xz1)(xy3)--(xz1)(xy3)--(xsx2)(xy4)--(xsx2)(xs1)--(xs2)(p2)--(p4)(p4)--(p5)(p2)--(p5)(xtx2)--(p1)(xz2)--(p1)(xz2)--(p2)(xtx1)--(p5)(xtx2)--(p2)(xtx1)--(t)(t)--(xtx2)(xtx1)--(xtx2)(xs1)--(p2);
			\end{tikzpicture}
		\end{minipage}
		\hspace{0.005\textwidth}
		\begin{minipage}[t]{.485\textwidth}%
			\centering
			\begin{tikzpicture}[node distance=\nodedistance]
			\clip(-.75\nodedistance,2.5\nodedistance) rectangle(5.75\nodedistance, -2.75\nodedistance);
			\node[vertex, yshift=.5\nodedistance] (s) {};
			\labelLeft{s}{$s$}
			\node[vertex, right of=s, xshift=4\nodedistance] (t) {};
			\labelRight{t}{$t$}
			\node[vertex, right of=s, xshift=1.5\nodedistance] (a) {};
			\node[above of=a, yshift=-.75\nodedistance] {$a$};

			\node[above of=s, yshift=.7\nodedistance, xshift=.5\nodedistance] (1a) {};
			\node[right of=1a, xshift=.5\nodedistance] (1b) {};
			\node[above of=t, yshift=.7\nodedistance, xshift=-.5\nodedistance] (1c) {};
			\begin{pgfonlayer}{background}
			\fill[brown!20] \convexpath{s,1a, 1b, 1c, t, a}{.3\nodedistance};
			\end{pgfonlayer}
			
			\node[vertex, above of=s, xshift=.4\nodedistance, color5, fill=color5] (sx1) {};
			\node[vertex, above of=s, yshift=-.4\nodedistance, xshift=.8\nodedistance, color5, fill=color5] (sx2) {};
			\node[vertex, above of=t, xshift=-.4\nodedistance] (tx1) {};
			\node[vertex, above of=t, yshift=-.4\nodedistance, xshift=-.8\nodedistance] (tx2) {};
			\node[above of=sx1, xshift=.1\nodedistance, yshift=-.3\nodedistance] (y1) {$H_1$};
			\node[vertex, above of=sx1, xshift=.4\nodedistance, yshift=-.7\nodedistance] (ry1) {};
			\node[vertex, right of=y1] (y2) {};
			\node[vertex, right of=sx2, xshift=-.2\nodedistance, yshift=-.2\nodedistance, color5, fill=color5] (y3) {};
			\node[vertex, above of=a, xshift=-.5\nodedistance, yshift=-.2\nodedistance] (z1) {};
			\node[vertex, above of=a, xshift=.5\nodedistance] (z2) {};
			\node[vertex, right of=sx1, yshift=.2\nodedistance] (y4) {};
			\node[vertex, right of=y2, yshift=.1\nodedistance] (s1) {};
			\node[vertex, right of=y2, yshift=-.3\nodedistance, xshift=-.2\nodedistance] (s2) {};
			\node[vertex, below of=tx2, xshift=-.4\nodedistance, yshift=.7\nodedistance] (q1) {};
			\node[vertex, above of=q1, xshift=-.2\nodedistance, yshift=-.5\nodedistance] (q2) {};
			\node[vertex, above of=q2, xshift=.6\nodedistance, yshift=-.2\nodedistance] (q3) {};
			\node[vertex, left of =q3, xshift=.4\nodedistance, yshift=.2\nodedistance] (q4) {};
			\node[vertex, below of=q4, xshift=.2\nodedistance, yshift=.5\nodedistance] (q5) {};
			
			\draw[edge] (z1) -- (z2);
			\draw[edge, color5] (s)--(a);
			\draw[edge, dotted] (sx2)--(s2)(t) -- (tx1) (t) -- (tx2)(a)--(t) (s) -- (sx1) (s) -- (sx2)(sx1) -- (sx2) (sx2) -- (z1) (sx2) -- (y3) (y3) -- (z1) (sx1)--(ry1)(ry1)--(y4)(y4)--(y2) (s2)--(y4) (y2) --(s1) (y2) -- (s2) (y3)--(a) (s) -- (y3) (t)--(q1) (q1)--(a) (q1)--(q2) (q2)--(z2) (s2)--(z2) (s2)--(z1) (q2)--(q5) (q5)--(q4) (q5)--(q3) (q5)--(s2) (q5)--(s1) (s1)--(q4) (tx1) --(q3) (tx1) --(q5) (tx2)--(q1) (tx2)--(q5) (q3)--(q4)(a)--(q2);

			\node[vertex, below of=s] (snew) {};
			\labelLeft{snew}{$s$}
			\node[vertex, below of=a] (anew) {};
			\node[above of=anew, yshift=-.6\nodedistance] {$a$};
			\node[vertex, below of=t] (tnew) {};
			\labelRight{tnew}{$t$}
			\node[vertex, below of=snew, yshift=-.3\nodedistance, xshift=.6\nodedistance] (xsx1) {};
			\node[vertex, below of=snew, yshift=.2\nodedistance, xshift=.8\nodedistance] (xsx2) {};
			\node[vertex, below of=tnew, xshift=-.2\nodedistance] (xtx1) {};
			\node[vertex, below of=tnew, yshift=.3\nodedistance, xshift=-.7\nodedistance] (xtx2) {};
			\node[vertex,below of=snew, xshift=-0\nodedistance] (xstar) {};
			\node[vertex, below of=xsx1, xshift=-.1\nodedistance, yshift=.4\nodedistance] (xy1) {};
			\node[vertex, right of=xy1] (xy2) {};
			\node[vertex, right of=xsx2, xshift=-.2\nodedistance, yshift=-.2\nodedistance] (xy3) {};
			\node[vertex, below of=anew, yshift=.3\nodedistance, xshift=-.4\nodedistance] (xz1) {};
			\node[vertex, below of=anew, yshift=.3\nodedistance, xshift=.4\nodedistance, color5, fill=color5] (xz2) {};
			\node[vertex, right of=xsx1, yshift=-.2\nodedistance] (xy4) {};
			\node[vertex, right of=xy2, yshift=.1\nodedistance, color5, fill=color5] (xs1) {};
			\node[vertex, right of=xy2, yshift=.6\nodedistance, xshift=-.4\nodedistance] (xs2) {};
			\node[vertex, above of=xtx2, xshift=-.5\nodedistance, yshift=-.9\nodedistance] (p1) {};
			\node[vertex, below of=p1, xshift=-.2\nodedistance, yshift=.4\nodedistance] (p2) {};
			\node[below of=p2, xshift=.8\nodedistance, yshift=.2\nodedistance] (p3) {$H_2$};
			\node[vertex, left of =p3, xshift=.0\nodedistance, yshift=0\nodedistance] (p4) {};
			\node[vertex, above of=p4, xshift=.8\nodedistance, yshift=-.4\nodedistance] (p5) {};
			
			\draw[edge, color5] (xs2)--(p2);
			\draw[edge, color5] (anew) --(tnew);
			\draw[edge, dotted] (xz1)--(xz2)(xsx1)--(xy4)(xs2)--(xy3)(xs1)--(p4) (tnew)--(p1)(xstar)--(snew)(xstar)--(xy1)(snew)--(xz1) (snew)--(xsx1)(snew)--(xsx2)(xsx1)--(xy1)(xsx2)--(xz1)(xz1)--(anew)(xz2)--(anew)(xy1)--(xy2)(xy2)--(xy4)(xy2)--(xs1)(xy2)--(xs2)(xy4)--(xs2)(xs2)--(xz2)(xs2)--(xz1)(xy3)--(xz1)(xy3)--(xsx2)(xy4)--(xsx2)(xs1)--(xs2)(p2)--(p4)(p4)--(p5)(p2)--(p5)(xtx2)--(p1)(xz2)--(p1)(xz2)--(p2)(xtx1)--(p5)(xtx2)--(p2)(xtx1)--(tnew)(tnew)--(xtx2)(xtx1)--(xtx2)(xs1)--(p2);
			
			\node[below of=snew, yshift=.3\nodedistance] (2a) {};
			\node[below of=snew, yshift=.3\nodedistance, xshift=.3\nodedistance] (2aa) {};
			\node[below of=anew, yshift=.3\nodedistance, xshift=-.3\nodedistance] (2ab) {};
			\node[below of=anew, yshift=.3\nodedistance, xshift=.3\nodedistance] (2bb) {};
			\node[below of=tnew, yshift=.3\nodedistance, xshift=-.3\nodedistance] (2bc) {};
			\node[below of=anew, yshift=.3\nodedistance] (2b) {};
			\node[below of=tnew, yshift=.3\nodedistance] (2c) {};
			\node[below of=2a, yshift=-.3\nodedistance, xshift=.5\nodedistance] (3a) {};
			\node[below of=2c, yshift=-.3\nodedistance, xshift=-.5\nodedistance] (3b) {};
			\begin{pgfonlayer}{background}
			\fill[black!20] \convexpath{snew,2aa, 2ab, anew, 2bb, 2bc, tnew, 2c, 3b, 3a, 2a, snew}{.3\nodedistance};
			\end{pgfonlayer}
			\end{tikzpicture}
		\end{minipage}
		\caption{Case~1 of the proof of Theorem~\ref{thm:conjecture_tw_3}: Partitioning the Graph $G$ into $G_1$ and $G_2$ on the left. Defining the graphs $H_1$ and $H_2$ on the right. Colored elements form $s$-$t$ disconnecting pair. Dotted lines indicate a connection between vertices.}
		\label{fig:cp_construction_tw3}
	\end{figure}
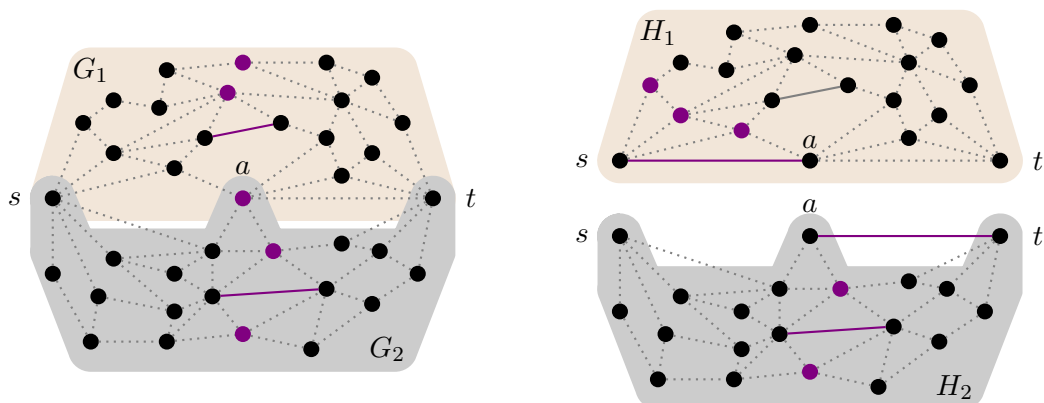
	
	Note that $\card{V(G_i)} <
	\card{V(G)}$ for $i\in\set{1,2}$. Further, as $B_x\cap B_y=\set{s,t,a}$, the two
	components of $\mathcal{T}-xy$ induce tree decompositions of $G_1$ and $G_2$ and
	we get $\tw(G_1)\leq 3$ and $\tw(G_2)\leq 3$. As the vertices~$s$, $t$ and $a$
	are contained in both $B_x$ and $B_y$, adding edges between these vertices in
	$G_1$ or $G_2$ does not increase the treewidth of the graphs. We distinguish two
	cases.

	\medskip
	
	\noindent\textbf{Case 1:} \textit{The vertex $a$ is contained in an $s$-$t$
		disconnecting pair of order $k$ and size $l$ in~$G$.}
	
	\smallskip
	
	In $G-a$ the pair $(k-1,l)$ is a connectivity pair for $s$ and $t$. Let $(W,F)$
	be an $s$-$t$ disconnecting pair of order $k$ and size $l$ with $a\in W$. For
	$i=1,2$, the pair $(W,F)$ induces an $s$-$t$ disconnecting pair $(W_i,F_i)$ in
	$G_i-a$ of order $k_i$ and size $l_i$ such that $k_1+k_2=k-1$ and $l_1+l_2=l$.
	Without loss of generality we may assume $l_1\geq 1$.
	
	\begin{subclaim}
		\label{sc:connectivity_pair}
		$(k_i,l_i)$ is a connectivity pair for $s$ and $t$ in $G_i-a$.
		\begin{proof}
			Suppose $(k_i,l_i)$ is not a connectivity pair. As $(W_i,F_i)$ is a 
			disconnecting pair of order $k_i$ and size $l_i$ such a pair 
			exists. Thus, there is a disconnecting pair $(W^\prime, 
			F^\prime)$ of order $k^\prime\leq k_i$, size $l^\prime\leq l_i$ and 
			cardinality $k^\prime+l^\prime<k_i+l_i$, but then $(W^\prime\cup 
			W_j\cup\{a\},F^\prime\cup F_j)$ with $j\in\{1,2\}\setminus \{i\}$ is an $s$-$t$ 
			disconnecting pair in $G$ of order at most $k$, size at most $l$ and 
			cardinality less than $k+l$ which yields a contradiction.
		\end{proof}
	\end{subclaim}
	Next we show that if $(k_2,l_2)$ is not a connectivity pair for $s$ and $t$ in
	$G_2$, the desired paths exist. If $(k_2,l_2)$ is in fact not a connectivity
	pair, then $(k_2, l_2 + p)$ is a connectivity pair for some integer~${p\geq 1}$.
	By induction we get $k_2+l_2+p$ edge-disjoint $s$-$t$ paths in $G_2$, $k_2+1$
	of which are internally disjoint. As $(k_1, l_1)$ is a connectivity pair
	in $G_1-a$ again by induction we get $k_1+l_1$ edge-disjoint $s$-$t$ paths in
	$G_1-a$ of which $k_1+1$ are internally disjoint (recall that $l_1\geq
	1$). Together we get $k_2 + l_2 + p + k_1 + l_1\geq k+l$ edge-disjoint $s$-$t$
	paths in $G$ of which $k_1+1+k_2+1=k+1$ are internally disjoint. Thus,
	we may assume that $(k_2,l_2)$ is a  connectivity pair for $s$ and $t$ in
	$G_2$.
	
	Now $a$ cannot be contained in any $s$-$t$ disconnecting pair in $G_2$ of order
	$k_2$ and size $l_2$, as otherwise $(k_2,l_2)$ would not be a connectivity pair
	for $s$ and $t$ in $G_2-a$. We fix some $s$-$t$ disconnecting pair
	$(W_2^\prime, F_2^\prime)$ in $G_2$ that has order $k_2$ and size~$l_2$. As in
	$G_2-W^\prime_2- F^\prime_2$, the vertex $a$ cannot be connected to both, $s$ and
	$t$, without loss of generality we may assume that $a$ is not connected to~$t$.
	Thus, for the remainder of Case~1 we assume
	\begin{align}
		\label{thm:conjecture_tw_3_sub1}
		\begin{minipage}{.9\textwidth}
			\textit{$(k_2,l_2)$ is a connectivity pair for $s$ and $t$ in 
				$G_2$, $(W_2^\prime, F_2^\prime)$ is an $s$-$t$ disconnecting pair 
				of order~$k_2$ and size~$l_2$ that does not contain $a$, and the 
				vertices $a$ and $t$ are not connected in 
				$G_2-W^\prime_2-F^\prime_2$.}
		\end{minipage}
	\end{align}
	
	We define $0\leq q\leq l_1$ to be the unique integer such that ${(k_1+1,
		l_1-q)}$ is a connectivity pair for $s$ and $t$ in $G_1$. Note that this is
	well-defined as $(W_1\cup\{a\},F_1)$ is an $s$-$t$ disconnecting pair of order
	$k_1+1$ and size $l_1$.
	
	Denote by $H_1$ the graph arising from $G_1$ by adding $q$ parallel edges
	$e_1,\dots,e_q$ between $a$ and $s$, \textit{cf.}~Figure~\ref{fig:cp_construction_tw3}. Then the following holds:
	\begin{subclaim}
		\label{sc:h1}
		$(k_1+1,l_1)$ is a connectivity pair for $s$ and $t$ in $H_1$.
		\begin{proof}
			As argued before, $(W_1\cup\{a\}, F_1)$ is an $s$-$t$ disconnecting pair of
			order $k_1+1$ and size $l_1$ in $G_1$ and thereby also disconnecting in $H_1$.
			Suppose that there exists an $s$-$t$ disconnecting pair of order $k_1+1$ and
			size at most $l_1-1$. Let $(W^{\prime}, F^{\prime})$ be one such pair of
			minimal size. Suppose that $e_i\in F^{\prime}$ for some $i\in\{1,\dots,q\}$. As
			the size of the disconnecting pair is minimal, this implies $e_i\in F^{\prime}$
			for all $i\in\{1,\dots,q\}$. If this is the case $(W^\prime,
			F^\prime\setminus\{e_1,\dots,e_q\})$ is an $s$-$t$ disconnecting pair in $G_1$
			of order at most $k_1+1$ and size at most $l_1-1-q$ in contradiction to
			$(k_1+1, l_1-q)$ being a connectivity pair in~$G_1$. Thus, either $a\in
			W^\prime$ or $a$ and $s$ are contained in the same component of
			$H_1-W^\prime-F^\prime$. In particular, there is no $a$-$t$ path in
			$G_1-W^\prime-F^\prime$. But then $(W^\prime\cup W_2^\prime, F^\prime\cup
			F_2^\prime)$ is an $s$-$t$ disconnecting pair in $G$ of order at most $k$ and
			size at most $l-1$ by \eqref{thm:conjecture_tw_3_sub1}. This contradicts
			$(k,l)$ being a connectivity pair for $s$ and $t$ in $G$.
		\end{proof}
	\end{subclaim} 
	Next, denote by $H_2$ the graph arising from $G_2$ by adding $q$ parallel edges
	$f_1,\dots, f_q$ between $a$ and $t$, \textit{cf.}~Figure~\ref{fig:cp_construction_tw3}. We prove the following:
	\begin{subclaim}
		\label{gt::sc:h2}
		$(k_2,l_2+q)$ is a connectivity pair for $s$ and $t$ in $H_2$.
		\begin{proof}
			If $q=0$ the statement holds by~\eqref{thm:conjecture_tw_3_sub1}, so
			assume that $q\geq 1$. The pair $(W_2^\prime,F_2^\prime\cup\{f_1,\dots,f_q\})$
			is an $s$-$t$ disconnecting pair in $H_2$ and $a$ and $t$ are not in the same
			component in $H_2-W_2'-F_2'\cup \{f_1,\ldots,f_q\}$. So suppose there exists an
			$s$-$t$ disconnecting pair of order $k_2$ and size $l_2+q-1$. Let  $(W^\prime,
			F^\prime)$ be one such pair of minimal size. With the same arguments as in the
			previous claim we get that $f_i\notin F^\prime$ for all $i\in\{1,\dots,q\}$.
			Thus, $a\in W^\prime$ or $a$ and $t$ are in the same component in
			$H_2-W^\prime-F^\prime$. In particular, there is no $s$-$a$ path in
			$G_2-W^\prime-F^\prime$. As $(k_1+1,l_1-q)$ is a connectivity pair for~$G_1$,
			there exists an $s$-$t$ disconnecting pair~$(W_1^\prime, F_1^\prime)$ in $G_1$
			of order $k_1+1$ and size~$l_1-q$. Suppose there exists an $s$-$a$ path in
			$G_1-W_1^\prime-F_1^\prime$. Then, there is no $a$-$t$ path and
			$(W_1^\prime\cup W_2^\prime, F_1^\prime\cup F_2^\prime)$ is an $s$-$t$
			disconnecting pair in~$G$ of order $k$ and size $l_2 + l_1 - q < l$
			by~\eqref{thm:conjecture_tw_3_sub1} \textemdash\xspace a contradiction. On the other hand if there
			does not exist an $s$-$a$ path in $G_1-W_1^\prime - F_1^\prime$, then the pair
			$(W_1^\prime\cup W^\prime, F_1^\prime\cup F^\prime)$ is $s$-$t$ disconnecting
			in $G$ and of order $k$ and size $l-1$, which again yields a contradiction.
		\end{proof}
	\end{subclaim}
	
	\begin{figure}[ht]
		\centering
		\begin{minipage}[t]{.485\textwidth}%
			\centering
			\begin{tikzpicture}[node distance=\nodedistance]
			\clip(-.75\nodedistance,2.5\nodedistance) rectangle(5.75\nodedistance, -2.75\nodedistance);
			\node[vertex, yshift=.5\nodedistance] (s) {};
			\labelLeft{s}{$s$}
			\node[vertex, right of=s, xshift=4\nodedistance] (t) {};
			\labelRight{t}{$t$}
			\node[vertex, right of=s, xshift=1.5\nodedistance] (a) {};
			\node[above of=a, yshift=-.75\nodedistance] {$a$};

			\node[above of=s, yshift=.7\nodedistance, xshift=.5\nodedistance] (1a) {};
			\node[right of=1a, xshift=.5\nodedistance] (1b) {};
			\node[above of=t, yshift=.7\nodedistance, xshift=-.5\nodedistance] (1c) {};
			\begin{pgfonlayer}{background}
			\fill[brown!20] \convexpath{s,1a, 1b, 1c, t, a}{.3\nodedistance};
			\end{pgfonlayer}
			
			\node[vertex, above of=s, xshift=.4\nodedistance] (sx1) {};
			\node[vertex, above of=s, yshift=-.4\nodedistance, xshift=.8\nodedistance] (sx2) {};
			\node[vertex, above of=t, xshift=-.4\nodedistance] (tx1) {};
			\node[vertex, above of=t, yshift=-.4\nodedistance, xshift=-.8\nodedistance] (tx2) {};
			\node[above of=sx1, xshift=.1\nodedistance, yshift=-.3\nodedistance] (y1) {$H_1$};
			\node[vertex, above of=sx1, xshift=.4\nodedistance, yshift=-.7\nodedistance] (ry1) {};
			\node[vertex, right of=y1] (y2) {};
			\node[vertex, right of=sx2, xshift=-.2\nodedistance, yshift=-.2\nodedistance] (y3) {};
			\node[vertex, above of=a, xshift=-.5\nodedistance, yshift=-.2\nodedistance] (z1) {};
			\node[vertex, above of=a, xshift=.5\nodedistance] (z2) {};
			\node[vertex, right of=sx1, yshift=.2\nodedistance] (y4) {};
			\node[vertex, right of=y2, yshift=.1\nodedistance] (s1) {};
			\node[vertex, right of=y2, yshift=-.3\nodedistance, xshift=-.2\nodedistance] (s2) {};
			\node[vertex, below of=tx2, xshift=-.4\nodedistance, yshift=.7\nodedistance] (q1) {};
			\node[vertex, above of=q1, xshift=-.2\nodedistance, yshift=-.5\nodedistance] (q2) {};
			\node[vertex, above of=q2, xshift=.6\nodedistance, yshift=-.2\nodedistance] (q3) {};
			\node[vertex, left of =q3, xshift=.4\nodedistance, yshift=.2\nodedistance] (q4) {};
			\node[vertex, below of=q4, xshift=.2\nodedistance, yshift=.5\nodedistance] (q5) {};
			
			\draw[edge, color2] (z1) -- (z2);
			\draw[edge, color2] (s)--(a);
			\draw[edge, color2, dotted] (a)--(t);
			\draw[edge, color2, dotted] (s)--(sx1)(sx1)--(ry1)(ry1)--(y4)(y4)--(y2)(y2) --(s1)(t) -- (tx1)(tx1) --(q3)(q3)--(q4)(s1)--(q4)  ;
			\draw[edge, color2, dotted] (s)--(sx2)(sx2)--(s2)(t) -- (tx2)(tx2)--(q5)(q5)--(s2);
			\draw[edge, color2, dotted] (s)--(y3)(y3) -- (z1)(t)--(q1)(q1)--(q2)(q2)--(z2) ;
			\draw[edge, dotted]   (sx1) -- (sx2) (sx2) -- (z1) (sx2) -- (y3)   (s2)--(y4)  (y2) -- (s2) (y3)--(a) (q1)--(a)   (s2)--(z2) (s2)--(z1) (q2)--(q5) (q5)--(q4) (q5)--(q3) (q5)--(s1)  (tx1) --(q5) (tx2)--(q1)  (a)--(q2);

			\node[vertex, below of=s] (snew) {};
			\labelLeft{snew}{$s$}
			\node[vertex, below of=a] (anew) {};
			\node[above of=anew, yshift=-.6\nodedistance] {$a$};
			\node[vertex, below of=t] (tnew) {};
			\labelRight{tnew}{$t$}
			\node[vertex, below of=snew, yshift=-.3\nodedistance, xshift=.6\nodedistance] (xsx1) {};
			\node[vertex, below of=snew, yshift=.2\nodedistance, xshift=.8\nodedistance] (xsx2) {};
			\node[vertex, below of=tnew, xshift=-.2\nodedistance] (xtx1) {};
			\node[vertex, below of=tnew, yshift=.3\nodedistance, xshift=-.7\nodedistance] (xtx2) {};
			\node[vertex,below of=snew, xshift=-0\nodedistance] (xstar) {};
			\node[vertex, below of=xsx1, xshift=-.1\nodedistance, yshift=.4\nodedistance] (xy1) {};
			\node[vertex, right of=xy1] (xy2) {};
			\node[vertex, right of=xsx2, xshift=-.2\nodedistance, yshift=-.2\nodedistance] (xy3) {};
			\node[vertex, below of=anew, yshift=.3\nodedistance, xshift=-.4\nodedistance] (xz1) {};
			\node[vertex, below of=anew, yshift=.3\nodedistance, xshift=.4\nodedistance] (xz2) {};
			\node[vertex, right of=xsx1, yshift=-.2\nodedistance] (xy4) {};
			\node[vertex, right of=xy2, yshift=.1\nodedistance] (xs1) {};
			\node[vertex, right of=xy2, yshift=.6\nodedistance, xshift=-.4\nodedistance] (xs2) {};
			\node[vertex, above of=xtx2, xshift=-.5\nodedistance, yshift=-.9\nodedistance] (p1) {};
			\node[vertex, below of=p1, xshift=-.2\nodedistance, yshift=.4\nodedistance] (p2) {};
			\node[below of=p2, xshift=.8\nodedistance, yshift=.2\nodedistance] (p3) {$H_2$};
			\node[vertex, left of =p3, xshift=.0\nodedistance, yshift=0\nodedistance] (p4) {};
			\node[vertex, above of=p4, xshift=.8\nodedistance, yshift=-.4\nodedistance] (p5) {};
			
			\draw[edge, color2] (xs2)--(p2);
			\draw[edge, color2] (anew) --(tnew);
			\draw[edge, dotted, color2](xstar)--(snew)(xstar)--(xy1)(xy1)--(xy2)(xy2)--(xs1)(xs1)--(p4)(p4)--(p5)(xtx1)--(tnew)(xtx2)--(p2)(tnew)--(xtx2)(xtx1)--(p5);
			\draw[edge, dashed, color2] (snew)--(xsx2)(xsx2)--(xz1)(xz1)--(xz2)(xz2)--(p1)(tnew)--(p1);
			\draw[edge, dotted, color2](snew)--(xz1)(xz1)--(anew);
			\draw[edge, dotted, color2](snew)--(xsx1)(xsx1)--(xy4)(xy4)--(xs2);
			\draw[edge, dotted] (xs2)--(xy3) (xsx1)--(xy1)(xz2)--(anew)(xy2)--(xy4)(xy2)--(xs2)(xs2)--(xz2)(xs2)--(xz1)(xy3)--(xz1)(xy3)--(xsx2)(xy4)--(xsx2)(xs1)--(xs2)(p2)--(p4)(p2)--(p5)(xtx2)--(p1)(xz2)--(p2)(xtx1)--(xtx2)(xs1)--(p2);
			
			\node[below of=snew, yshift=.3\nodedistance] (2a) {};
			\node[below of=snew, yshift=.3\nodedistance, xshift=.3\nodedistance] (2aa) {};
			\node[below of=anew, yshift=.3\nodedistance, xshift=-.3\nodedistance] (2ab) {};
			\node[below of=anew, yshift=.3\nodedistance, xshift=.3\nodedistance] (2bb) {};
			\node[below of=tnew, yshift=.3\nodedistance, xshift=-.3\nodedistance] (2bc) {};
			\node[below of=anew, yshift=.3\nodedistance] (2b) {};
			\node[below of=tnew, yshift=.3\nodedistance] (2c) {};
			\node[below of=2a, yshift=-.3\nodedistance, xshift=.5\nodedistance] (3a) {};
			\node[below of=2c, yshift=-.3\nodedistance, xshift=-.5\nodedistance] (3b) {};
			\begin{pgfonlayer}{background}
			\fill[black!20] \convexpath{snew,2aa, 2ab, anew, 2bb, 2bc, tnew, 2c, 3b, 3a, 2a, snew}{.3\nodedistance};
			\end{pgfonlayer}
			\end{tikzpicture}
		\end{minipage}
		\hspace{0.005\textwidth}
		\begin{minipage}[t]{.485\textwidth}%
			\centering
			\begin{tikzpicture}[node distance=\nodedistance]
			\clip(-.75\nodedistance,2.5\nodedistance) rectangle(5.75\nodedistance, -2.75\nodedistance);
			\node[vertex] (s) {};
			\labelLeft{s}{$s$}
			\node[vertex, right of=s, xshift=4\nodedistance] (t) {};
			\labelRight{t}{$t$}
			\node[vertex, right of=s, xshift=1.5\nodedistance] (a) {};
			\node[above of=a, yshift=-.6\nodedistance] {$a$};

			\node[above of=s, yshift=.7\nodedistance, xshift=.5\nodedistance] (1a) {};
			\node[right of=1a, xshift=.5\nodedistance] (1b) {};
			\node[above of=t, yshift=.7\nodedistance, xshift=-.5\nodedistance] (1c) {};
			\begin{pgfonlayer}{background}
			\fill[brown!20] \convexpath{s,1a, 1b, 1c, t, a}{.3\nodedistance};
			\end{pgfonlayer}
			
			\node[below of=s, yshift=.3\nodedistance] (2a) {};
			\node[below of=s, yshift=.3\nodedistance, xshift=.3\nodedistance] (2aa) {};
			\node[below of=a, yshift=.3\nodedistance, xshift=-.3\nodedistance] (2ab) {};
			\node[below of=a, yshift=.3\nodedistance, xshift=.3\nodedistance] (2bb) {};
			\node[below of=t, yshift=.3\nodedistance, xshift=-.3\nodedistance] (2bc) {};
			\node[below of=a, yshift=.3\nodedistance] (2b) {};
			\node[below of=t, yshift=.3\nodedistance] (2c) {};
			\node[below of=2a, yshift=-.3\nodedistance, xshift=.5\nodedistance] (3a) {};
			\node[below of=2c, yshift=-.3\nodedistance, xshift=-.5\nodedistance] (3b) {};
			\begin{pgfonlayer}{background}
			\fill[black!20] \convexpath{s,2aa, 2ab, a, 2bb, 2bc, t, 2c, 3b, 3a, 2a, s}{.3\nodedistance};
			\end{pgfonlayer}
			
			\node[vertex, above of=s, xshift=.4\nodedistance] (sx1) {};
			\node[vertex, above of=s, yshift=-.4\nodedistance, xshift=.8\nodedistance] (sx2) {};
			\node[vertex, above of=t, xshift=-.4\nodedistance] (tx1) {};
			\node[vertex, above of=t, yshift=-.4\nodedistance, xshift=-.8\nodedistance] (tx2) {};
			\node[above of=sx1, xshift=.1\nodedistance, yshift=-.3\nodedistance] (y1) {$G_1$};
			\node[vertex, above of=sx1, xshift=.4\nodedistance, yshift=-.7\nodedistance] (ry1) {};
			\node[vertex, right of=y1] (y2) {};
			\node[vertex, right of=sx2, xshift=-.2\nodedistance, yshift=-.2\nodedistance] (y3) {};
			\node[vertex, above of=a, xshift=-.5\nodedistance, yshift=-.2\nodedistance] (z1) {};
			\node[vertex, above of=a, xshift=.5\nodedistance] (z2) {};
			\node[vertex, right of=sx1, yshift=.2\nodedistance] (y4) {};
			\node[vertex, right of=y2, yshift=.1\nodedistance] (s1) {};
			\node[vertex, right of=y2, yshift=-.3\nodedistance, xshift=-.2\nodedistance] (s2) {};
			\node[vertex, below of=tx2, xshift=-.4\nodedistance, yshift=.7\nodedistance] (q1) {};
			\node[vertex, above of=q1, xshift=-.2\nodedistance, yshift=-.5\nodedistance] (q2) {};
			\node[vertex, above of=q2, xshift=.6\nodedistance, yshift=-.2\nodedistance] (q3) {};
			\node[vertex, left of =q3, xshift=.4\nodedistance, yshift=.2\nodedistance] (q4) {};
			\node[vertex, below of=q4, xshift=.2\nodedistance, yshift=.5\nodedistance] (q5) {};
			
			\draw[edge, color2] (z1) -- (z2);
			\draw[edge, color2, dotted] (a)--(t);
			\draw[edge, color2, dotted] (s)--(sx1)(sx1)--(ry1)(ry1)--(y4)(y4)--(y2)(y2) --(s1)(t) -- (tx1)(tx1) --(q3)(q3)--(q4)(s1)--(q4)  ;
			\draw[edge, color2, dotted] (s)--(sx2)(sx2)--(s2)(t) -- (tx2)(tx2)--(q5)(q5)--(s2);
			\draw[edge, color2, dotted] (s)--(y3)(y3) -- (z1)(t)--(q1)(q1)--(q2)(q2)--(z2) ;
			\draw[edge, dotted]   (sx1) -- (sx2) (sx2) -- (z1) (sx2) -- (y3)   (s2)--(y4)  (y2) -- (s2) (y3)--(a) (q1)--(a)   (s2)--(z2) (s2)--(z1) (q2)--(q5) (q5)--(q4) (q5)--(q3) (q5)--(s1)  (tx1) --(q5) (tx2)--(q1)  (a)--(q2);
			
			\node[vertex, below of=s, yshift=-.3\nodedistance, xshift=.6\nodedistance] (xsx1) {};
			\node[vertex, below of=s, yshift=.2\nodedistance, xshift=.8\nodedistance] (xsx2) {};
			\node[vertex, below of=t, xshift=-.2\nodedistance] (xtx1) {};
			\node[vertex, below of=t, yshift=.3\nodedistance, xshift=-.7\nodedistance] (xtx2) {};
			\node[vertex,below of=s, xshift=-0\nodedistance] (xstar) {};
			\node[vertex, below of=xsx1, xshift=-.1\nodedistance, yshift=.4\nodedistance] (xy1) {};
			\node[vertex, right of=xy1] (xy2) {};
			\node[vertex, right of=xsx2, xshift=-.2\nodedistance, yshift=-.2\nodedistance] (xy3) {};
			\node[vertex, below of=a, yshift=.3\nodedistance, xshift=-.4\nodedistance] (xz1) {};
			\node[vertex, below of=a, yshift=.3\nodedistance, xshift=.4\nodedistance] (xz2) {};
			\node[vertex, right of=xsx1, yshift=-.2\nodedistance] (xy4) {};
			\node[vertex, right of=xy2, yshift=.1\nodedistance] (xs1) {};
			\node[vertex, right of=xy2, yshift=.6\nodedistance, xshift=-.4\nodedistance] (xs2) {};
			\node[vertex, above of=xtx2, xshift=-.5\nodedistance, yshift=-.9\nodedistance] (p1) {};
			\node[vertex, below of=p1, xshift=-.2\nodedistance, yshift=.4\nodedistance] (p2) {};
			\node[below of=p2, xshift=.8\nodedistance, yshift=.2\nodedistance] (p3) {$G_2$};
			\node[vertex, left of =p3, xshift=.0\nodedistance, yshift=0\nodedistance] (p4) {};
			\node[vertex, above of=p4, xshift=.8\nodedistance, yshift=-.4\nodedistance] (p5) {};
			
			\draw[edge, color2] (xs2)--(p2);
			\draw[edge, dotted, color2](xstar)--(s)(xstar)--(xy1)(xy1)--(xy2)(xy2)--(xs1)(xs1)--(p4)(p4)--(p5)(xtx1)--(t)(xtx2)--(p2)(t)--(xtx2)(xtx1)--(p5);
			\draw[edge, dashed, color2] (s)--(xsx2)(xsx2)--(xz1)(xz1)--(xz2)(xz2)--(p1)(t)--(p1);
			\draw[edge, dotted, color2](s)--(xz1)(xz1)--(a);
			\draw[edge, dotted, color2](s)--(xsx1)(xsx1)--(xy4)(xy4)--(xs2);
			\draw[edge, dotted] (xs2)--(xy3) (xsx1)--(xy1)(xz2)--(a)(xy2)--(xy4)(xy2)--(xs2)(xs2)--(xz2)(xs2)--(xz1)(xy3)--(xz1)(xy3)--(xsx2)(xy4)--(xsx2)(xs1)--(xs2)(p2)--(p4)(p2)--(p5)(xtx2)--(p1)(xz2)--(p2)(xtx1)--(xtx2)(xs1)--(p2);
			\end{tikzpicture}
		\end{minipage}
		\caption{Case~1 of the proof of Theorem~\ref{thm:conjecture_tw_3}: $s$-$t$ paths in $H_1$ and $H_2$ given by induction on the left. Creation of desired $s$-$t$ paths in $G$ on the right. Colored dotted lines indicate internally disjoint $s$-$t$ paths. Dashed line indicates an $s$-$t$ path that is edge-disjoint to all other indicated paths. Solid lines are single edges.}
		\label{fig:cp_paths_tw3}
	\end{figure}
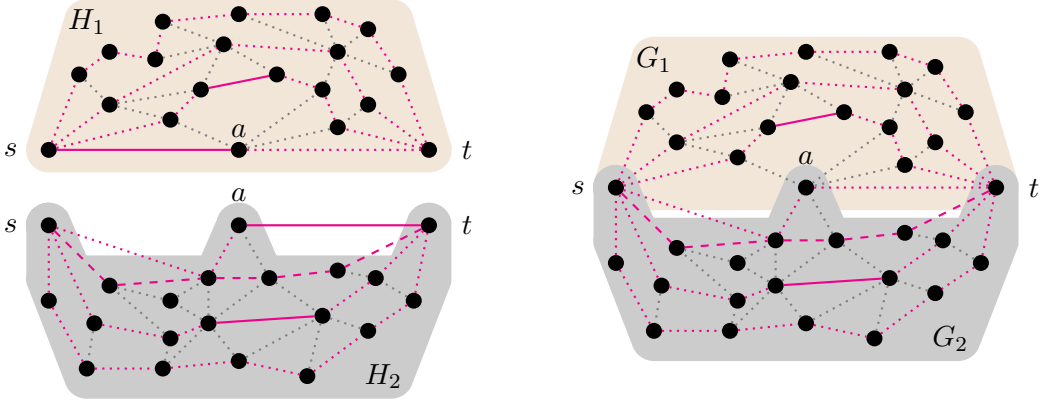

	Note that for $i\in\set{1,2}$ it is $V(H_i)=V(G_i)< V(G)$ and $\tw(H_i)\leq 3$ as
	$\tw(G_i)\leq 3$ and we only added edges between $s$ and $a$, respectively $a$
	and $t$ to get to $H_i$ from $G_i$. By Claim~\ref{sc:h1} and the induction
	hypothesis there are $k_1+1+l_1$ edge-disjoint $s$-$t$ paths in $H_1$, say
	$P_1,\dots, P_{k_1+l_1+1}$, of which $k_1+2$ are internally disjoint, \textit{cf.}~Figure~\ref{fig:cp_paths_tw3}.
	Without loss of generality let $P_1,\dots,P_{r_1}$ be the paths using edges from
	$\{e_1,\dots, e_q\}$, where from these we denote by $P_1$ the path that is among
	the $k_1+2$ internally disjoint paths, if one such path exists. If
	$q=l_2=0$, then $(k_2,0)$ is a connectivity pair for $s$ and $t$ in $G_2-a$ by
	Claim~\ref{sc:connectivity_pair}, and by
	Observation~\ref{obs:bineke_harary_l=1_or_k=0} there are $k_2$ internally
	disjoint $s$-$t$ paths in $G_2-a$. Together with
	$P_1,\dots,P_{k_1+l_1+1}$ we get the desired paths for $G$. So assume that
	$q+l_2>0$. By Claim~\ref{gt::sc:h2} and the induction hypothesis there are $k_2+l_2+q$
	edge-disjoint $s$-$t$ paths in~$H_2$, say $Q_1,\dots,Q_{k_2+l_2+q}$ of which
	$k_2+1$ are internally disjoint, \textit{cf.}~Figure~\ref{fig:cp_paths_tw3}. Without loss of generality let
	$Q_1,\dots,Q_{r_2}$ for $r_2\leq q$ be the paths using edges from
	$\{f_1,\dots,f_q\}$, where again from these we denote by $Q_1$ the path that is
	among the $k_2+1$ internally disjoint paths, if one such path exists. We
	now claim that for $r:=\min\{r_1,r_2\}$ the paths
	\begin{align*}
		Q_1a\cup aP_1,\dots, Q_ra\cup aP_r, P_{r_1+1},\dots, P_{k_1+l_1+1}, 
		Q_{r_2+1},\dots,Q_{k_2+l_2+q}
	\end{align*} are at least $k+l$ edge-disjoint $s$-$t$ paths of which 
	at least $k+1$ are internally disjoint, \textit{cf.}~Figure~\ref{fig:cp_paths_tw3}. First note, that the number of 
	paths is exactly \begin{align*}
		r+\left(k_1+l_1+1\right)-r_1 +\left(k_2+l_2+q\right) -r_2 = k+l+q+r 
		-r_1 - r_2.
	\end{align*}
	As $r$ is equal to $r_i$ for some $i$ and $q$ is greater than or equal to $r_1$
	and $r_2$ we get that the number of paths is at least $k+l$. To see that among
	the paths above there are at least $k+1$ internally disjoint paths,
	note that we started off with a set of $k_1+2+k_2+1=k+2$ internally
	disjoint paths
	$\mathcal{P}\subseteq\{P_1,\dots,P_{k_1+l_1+1},Q_1,\dots,Q_{k_2+l_2+q}\}$.
	
	The only vertex besides $s$ and $t$ that may be contained in more than one
	path of $\mathcal{P}$ is~$a$. If $Q_1, P_1\in\mathcal{P}$ they are glued
	together and $k+1$ internally disjoint paths still remain. If only one
	of $P_1$ and $Q_1$, say $P_1$, is among the internally disjoint paths,
	then $\mathcal{P}\setminus\{P_1\}$ is a set of $k+1$ internally disjoint
	paths, as only one other path than $P_1$ may contain $a$. Finally if neither
	$P_1$ nor $Q_1$ are among the internally disjoint paths, then
	$\mathcal{P}$ contains a subset of internally disjoint paths of size
	$k+1$ as at most two paths in $\mathcal{P}$ may contain~$a$. This concludes
	Case~1.
	
	\medskip
	
	\noindent\textbf{Case 2:} \textit{The vertex $a$ is not contained in any
		$s$-$t$ disconnecting pair of order $k$ and size~$l$.}
	
	\smallskip
	
	Denote by $(W,F)$ an $s$-$t$ disconnecting pair of order $k$ and size $l$ and
	for $i\in\{1,2\}$ let $W_i =V(G_i)\cap W $, $k_i=\card{W_i}$, $F_i=E(G_i)\cap
	E_i$, and $l_i=\card{F_i}$. Then $k_1+k_2=k$ and $l_1+l_2=l$. Without loss of
	generality we may assume that there is no $s$-$a$ path in $G-W-F$ and thereby
	also no $s$-$a$ path in $G_i-W_i-F_i$ for $i\in\{1,2\}$.
	
	For $i\in\{1,2\}$ denote by $0\leq q_i\leq l_i$ the unique integer
	such that $(k_i,l_i-q_i)$ is a connectivity pair for $s$ and $t$ in~$G_i$. Note
	that this is well-defined as $(W_i, F_i)$ is an $s$-$t$ disconnecting pair
	in~$G_i$. We define $q=\max\{q_1,q_2\}$ and assume that $q=q_1$ (the case $q=q_2$ follows analogously).
	Let $(W_1^\prime, F_1^\prime)$ be an $s$-$t$ disconnecting pair
	in $G_1$ of order $k_1$ and size $l_1-q$ and denote by $H_1$ the graph arising
	from $G_1$ by adding $q$ edges $e_1,\dots ,e_{q}$ between $a$ and $t$.
	\begin{subclaim}
		\label{sc:connectivity_pair_h1} $(k_1,l_1)$ is a connectivity pair for $s$
		and $t$ in $H_1$.
		\begin{proof}
			If $q=0$ the claim holds by definition of $q$. So assume $q\geq 1$.
			Clearly $(W_1^\prime, F_1^\prime\cup \{e_1,\dots, e_{q}\})$ is an $s$-$t$
			disconnecting pair in $H_1$ of order $k_1$ and size~$l_1$. So suppose there
			exists an $s$-$t$ disconnecting pair of order $k_1$ and size at most $l_1
			-1$. Let $(W^\prime, F^\prime)$ be such a pair of minimal size. If
			$e_1,\dots, e_q\in F^\prime$ the pair $(W^\prime,
			F^\prime\setminus\{e_1,\dots, e_q\})$ is $s$-$t$ disconnecting in $G_1$ and
			of order $k_1$ and size at most $l_1-q-1$, contradicting the fact that
			$(k_1,l_1-q)$ is a connectivity pair for $s$ and $t$ in $G_1$. Thus, either
			$a\in W^\prime$ or $a$ and $t$ are contained in the same component in
			$H_1-W^\prime - F^\prime$. In particular, there is no $s$-$a$ path in $G_1 -
			W^\prime - F^\prime$ and thereby the pair $(W^\prime\cup W_2, F^\prime\cup
			F_2)$ is $s$-$t$ disconnecting in $G$ and of order $k$ and size at most~$l-1$.
			A contradiction to $(k,l)$ being a connectivity pair for $s$ and $t$
			in $G$.
		\end{proof}
	\end{subclaim} 
	Let now $H_2$ be the graph arising from $G_2$ by adding $q$ edges $f_1,\dots,
	f_{q}$ between $a$ and $s$. For $H_2$ we can also find a connectivity pair.
	
	\begin{subclaim}
		\label{sc:connectivity_pair_h2}
		$(k_2, l_2 + q)$ is a connectivity pair for $H_2$.
		\begin{proof}
			Again, if $q=0$ the claim is immediate by definition of $q$. So let $q\geq 1$.
			Then $(W_2, F_2\cup\{f_1,\dots, f_{q}\})$ is an $s$-$t$ disconnecting pair of
			order $k_2$ and size $l_2+q$ in~$H_2$. So suppose there exists an $s$-$t$
			disconnecting pair of order $k_2$ and size at most $l_2+q-1$. Let $(W^\prime,
			F^\prime)$ be such a pair of minimal size. If $f_1,\dots, f_{q}\in F^\prime$,
			then there is no $s$-$a$ path in $H_2-W^\prime-F^\prime$. This implies that
			$(W^\prime\cup W_1, F^\prime\setminus \{f_1,\dots,f_{q^\prime}\}\cup F_1)$ is
			a disconnecting pair in $G$ of order at most $k$ and size at most $l-1$
			yielding a contradiction. Thus, either $a\in W^\prime$ or $a$ and $s$ are
			contained in the same component in $H_2-W^\prime -F^\prime$. In particular,
			there is no $a$-$t$ path in $G_2-W^\prime - F^\prime$. If there is also no
			$a$-$t$ path in $G_1-W_1^\prime-F_1^\prime$, the pair $(W^\prime\cup
			W_1^\prime, F^\prime\cup F_1^\prime)$ is disconnecting in $G$ and of order at
			most $k$ and size at most $l-1$. This yields a contradiction to $(k,l)$ being
			a connectivity pair for $s$ and $t$ in $G$. So suppose that there is an
			$a$-$t$ path in~$G_1-W_1^\prime -F_1^\prime$. Then there is no $s$-$a$ path
			in~$G_1-W_1^\prime-F_1^\prime$ and $(W_1^\prime\cup W_2, F_1^\prime\cup F_2)$ is
			an $s$-$t$ disconnecting pair in $G$, that has order at most $k$ and size at
			most $l_1-q+l_2<l$ as $q\geq 1$. Again this contradicts the fact that $(k,l)$
			is a connectivity pair for $s$ and $t$ in $G$.
		\end{proof}
	\end{subclaim}
	As in the proof of Case~1 we use the induction hypothesis on $H_1$ and $H_2$ to
	get the desired paths. If neither $l_1=0$ nor $l_2+q=0$ we get the paths in $G$
	in the same manner as in Case~1 and therefore do not repeat the arguments here.
	
	For the other case, let $l_1^\prime=l_1$ and $l_2^\prime=l_2+q$. If we can show
	for $i\in\{1,2\}$, that if $l_i^\prime=0$, then in $G_i-a$ the pair $(k_i,0)$
	is a connectivity pair, we can again proceed as in Case~1 and get the desired
	paths. To see this we simply observe that $a$ is not contained in any
	$k_i$-vertex separator in $G_i$ as this would imply that $a$ is contained in an
	$s$-$t$ disconnecting pair of order $k$ and size $l$ in~$G$.
\end{proof}

\section{Complexity of Computing the Second Coordinate in a Connectivity Pair}
\label{sec:complexity}

In \cite{beineke2012topics} Oellermann posed the question how
difficult it is to compute the second coordinate in a connectivity pair.
In this section we deal
with this question and show that there is no polynomial time algorithm for this
problem unless $\textsf{P}=\textsf{NP}$. We formulate the question as the following decision
problem:

\begin{samepage}
\begin{decisionproblem}
	\problemtitle{\secondCoordP (\scp)}
	\probleminput{A graph $G$, distinct vertices $s,t\in V(G)$ and integers $k,B\in \mathbb{N}$. We denote such an instance by $(G,\{s,t\},k,B)$.}
	\problemquestion{Is there an $l\leq B$ such that $(k,l)$ is a connectivity pair for $s$ and~$t$ in~$G$?}
\end{decisionproblem}
\end{samepage}

Recall that for a graph $G$ with distinct, non-adjacent vertices $s$ and $t$ the value $\kappa_{G}(s,t)$ is the minimum number of vertices separating~$s$ and~$t$.
It can be readily observed, that for a fixed $k$ the second coordinate in a 
connectivity pair is unique in the following sense:
\begin{observation}
	\label{obs:second_coordinate_unique}
	Let $G$ be a graph with distinct
	vertices $s,t\in V(G)$ and let, $k,l\in\mathbb{N}$ with $0\leq k\leq
	\kappa_{G-E(s,t)}(s,t)$. There exists a unique integer $l_k$ such that
	$(k,l_k)$ is a connectivity pair for $s$ and $t$ in $G$. In particular, if
	there is an $s$-$t$ disconnecting pair of order $k$ and size $l$, then
	there exists a unique integer $l_k\leq l$ such that $(k,l_k)$ is a connectivity
	pair.
\end{observation}

To show \textsf{NP}-completeness of \scp we use \textsc{Bipartite Partial Vertex
	Cover} as a reduction partner. Let $G$ be a graph and $q$ a non-negative
integer. A set $C\subseteq V(G)$ is a \emph{vertex cover} if every edge of $G$
is incident with a vertex in $C$. Further, we call $C'\subseteq V(G)$ a
\emph{partial vertex cover} with respect to $q$ if at least $q$ edges in $E(G)$
are incident to a vertex in $C'$.

\begin{samepage}
\begin{decisionproblem}
	\problemtitle{\textsc{Bipartite Partial Vertex Cover} (\parvc)}
	\probleminput{A bipartite graph $G$ and integers $q,B\in \mathbb{N}$.
		We denote such an instance by $(G,q,B)$.}
	\problemquestion{Is there a $C\subseteq V(G)$ with $|C|\leq B$ such that $C$ is a partial vertex cover of $G$ with respect to $q$?}
\end{decisionproblem}
\end{samepage}

In \cite{caskurlu} Caskurlu and Subramani showed that \parvc is \textsf{NP}-hard.
Now we can prove the following result:

\begin{theorem}\label{thm:CSC_np}
	\secondCoordP is \textsf{NP}-complete.
\end{theorem}
\begin{proof}
	We begin by proving that \secondCoordP is \textsf{NP}-hard.
	Let $(H,q,B)$ be an instance of \parvc.
	In particular, $H=(U\dot{\cup} V,E(H))$ is a bipartite graph whose partition has the parts $U$ and $V$.
	With the help of König's Theorem, \textit{cf.}~\cite{schrijver-book}, and a maximum flow algorithm (for example the Hopcroft-Karp algorithm, \textit{cf.}~\cite{Hopcroft1973}) we can compute a minimum size vertex cover on bipartite graphs in polynomial time.
	Therefore,	we may assume that any vertex cover of $G$ has cardinality of at least $B+1$.
	
	We construct an instance $(G,\{s,t\},k,l)$ of \scp as follows:
	Set $V(G):= V(H) \cup \{s,t\}$ and let $E(G)$ consists of the set $E(H)$
	and for every $u\in U$ add $|E(H)|$ parallel edges between $s$ and $u$ and for every
	$v\in V$ add $|E(H)|$ parallel edges between $t$ and $v$.
	Moreover set $k:= B$ and $l:=|E(G)|-q$.
	The construction of $G$ is illustrated in Figure~\ref{tikz::beinecke_harary:complexity_cpp}.
	All the steps can be realized in polynomial time.
	Note that $s$ and $t$ are at least $(k+1)$-vertex-connected since a vertex
	cover in $G$ has cardinality at least $B+1>k$.

	\begin{figure}[ht]
		
		\centering
		\begin{tikzpicture}
		\draw
		(0,0) node[rectangle,rounded corners=15,draw,minimum width=145pt,minimum height=110pt,fill=black!10] (gg) {}
		;
		\draw[inner sep=1pt, minimum size = 20pt]
		(-2,1) node[vertex] (u1) {}
		(-2,0.1) node[circle] (u) {\vdots}
		(-2,-1) node[vertex] (u2) {}
		(2,1) node[vertex] (v1) {}
		(2,0.1) node[circle] (v) {\vdots}
		(2,-1) node[vertex] (v2) {}
		(-4,0) node[vertex] (s) {}
		(4,0) node[vertex] (t) {}
		(0,-1.2) node (g) {\large$H$}
		;
		\labelLeft{s}{$s$}
		\labelRight{t}{$t$}
		\labelAbove{u1}{$u_1$}
		\labelAbove{v1}{$v_1$}
		\labelBelow{v2}{$v_{|V|}$}
		\labelBelow{u2}{$v_{|U|}$}
		\draw[dashed]
		(u1) edge (v1)
		(u1) edge (v)
		(u1) edge (v2)
		(u) edge (v1)
		(u) edge (v)
		(u) edge (v2)
		(u2) edge (v1)
		(u2) edge (v)
		;
		\draw
		(-3,.61) node[rotate=26.6] (p) {$\vdots$}
		(-3,-.39) node[rotate=-26.6] (p) {$\vdots$}
		(3,.61) node[rotate=-26.6] (p) {$\vdots$}
		(3,-.39) node[rotate=26.6] (p) {$\vdots$}
		;
		\draw
		(s) edge[bend left=15] (u1)
		(s) edge[bend right=15] (u1)
		(s) edge[bend left=15] (u2)
		(s) edge[bend right=15] (u2)
		(t) edge[bend left=15] (v1)
		(t) edge[bend right=15] (v1)
		(t) edge[bend left=15] (v2)
		(t) edge[bend right=15] (v2)
		;
		\end{tikzpicture}
		\caption{The constructed graph $G$ in the proof of Theorem~\ref{thm:CSC_np}.}
		\label{tikz::beinecke_harary:complexity_cpp}
	\end{figure}
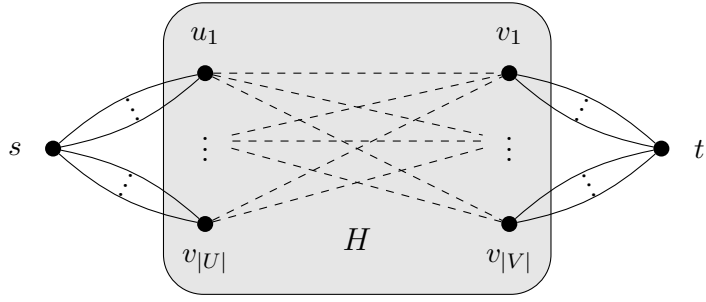
	
	We show that there is a partial vertex cover $C$ of $H$ with respect to $q$ and $|C|\leq B$ if and only if there is a connectivity pair $(k,l')$ for~$s$ and~$t$ in $G$ with~$l'\leq l$.
	
	First assume that there is a partial vertex cover $C$ of $H$ with respect to $q$ such that $|C|\leq B$.
	The pair $(C,F)$, where $F=\{uv\in E(H)\colon u,v\notin C \}$ is an $s$-$t$ disconnecting pair in $G$ with $|C|\leq B=k$ and $|F|\leq |E(H)|-q=l$.
	By Observation~\ref{obs:second_coordinate_unique} there is an $l_k\leq l$ such that $(k,l_k)$ is a connectivity pair for $s$ and $t$ in $G$.
	
	Conversely, assume that $(k,l')$ is a connectivity pair for $s$ and $t$ in $G$ for some~$l'\leq l$.
	Hence, there is an $s$-$t$ disconnecting pair $(W,F)$ of order $k$ and size $l'$ but non of order~$k$ and size less than $l'$.
	Therefore, we conclude that $F\subseteq E(H)$, since $l'\leq l\leq |E(H)|$ and an edge is in $F$ if and only if all its parallels are in $F$.
	Thus, we have for every edge $uv\in E(H)\setminus F$ that $u\in W$ or $v\in W$, otherwise $suvt$ is an $s$-$t$ path in $G-W-F$.
	Hence, the vertices in $W$ are incident to at least $|E(H)|-|F|\geq |E(H)|-l= q$ edges in~$H$.
	This shows that $W$ is a partial vertex cover of $H$ with respect to $q$ and $|W|=k=B$.
	
	Thus, \secondCoordP is \textsf{NP}-hard. To show that \scp in \textsf{NP}, let
$(G,\{s,t\},k,B)$ be an instance of \scp. If there exists an integer $l\leq B$
such that $(k,l)$ is a connectivity pair, there also exists an $s$-$t$
disconnecting pair of order $k$ and size at most $B$. Given a pair
$(W,F)$, where $W$ is a set of $k$ vertices and $F$ is a set of at most
$B$ edges, we may verify in polynomial time if $(W,F)$ is disconnecting. By
Observation~\ref{obs:second_coordinate_unique}, this implies that there exists
an integer $l\leq B$ such that $(k,l)$ is a connectivity pair.
	
	The arguments above may seem cumbersome, however, note that given a
disconnecting pair of order $k$ and size $l$, we may not verify in polynomial
time whether $(k,l)$ is a connectivity pair, unless $\textsf{P}=\textsf{NP}$:
Doing this would imply being able to decide in polynomial time if there is a
disconnecting pair of order $k-1$ and size $l$ as well as if there is a
disconnecting pair of order $k$ and size $l-1$. This, in turn, contradicts the
\textsf{NP}-hardness of \scp, unless $\textsf{P}=\textsf{NP}$.
\end{proof}

As a direct consequence of Theorem~\ref{thm:CSC_np} we get the intractability of the computation of the second coordinate of a
connectivity pair: An algorithm computing the second
coordinate in polynomial time could clearly be used to decide \scp in polynomial
time.

\begin{corollary}
	Unless \textsf{P}$=$\textsf{NP}, there is no polynomial time algorithm 
	that, given any graph $G$, distinct vertices $s$ and $t$ and an integer 
	$0\leq k\leq \kappa_{G-E(s,t)}(s,t)$, returns the integer $l$ such that 
	$(k,l)$ is a connectivity pair for $s$ and $t$ in $G$.
\end{corollary}


\section{Conclusion and Open Problems}
\label{sec:conclusion}

In this article we considered a form of mixed connectivity in graphs introduced
by Beineke and Harary, namely connectivity pairs. We prove the Beineke Harary
Conjecture for the case that $l=2$. This result substantially differs from
previous results in the literature and can be used to prove the conjecture on
restricted graph classes. We illustrate the latter fact by proving the
conjecture for graphs of treewidth at most $3$. From our studies the
Beineke-Harary-Conjecture may hold:
\begin{conjecture*}[Beineke-Harary-Conjecture]
	Let $G$ be a graph, $s,t\in V(G)$ distinct vertices and $k,l$ non-negative
integers with $l\geq 1$. If $(k,l)$ is a connectivity pair for $s$ and $t$ in
$G$, then there exist $k+l$ edge-disjoint paths, of which $k+1$ are internally
disjoint.
\end{conjecture*}

 We further proved the intractability of the computation of the second component
in a connectivity pair. This question has been mentioned to be interesting and
open by Oellermann in~\cite{beineke2012topics}. Observe that if $k$ is considered
to be constant we can find the integer $l_k$ in polynomial time: For each
subset $S$ of $k$ vertices, determine the edge-connectivity of $s$ and $t$ in
$G-S$. The smallest among them is the desired integer $l_k$. This, of
course, implies that whenever we assume that $s$ and $t$ can be separated by a
constant number of vertices or edges, we can also determine all connectivity
pairs for $s$ and $t$. One might think that this implies polynomial time
solvability on treewidth bounded graphs. However, observe that in a graph of
constant treewidth, the vertex-connectivity between a pair of vertices may very
well not be constant. Thus, polynomial time solvability on treewidth bounded
graphs is not directly implied and the complexity of this problem remains open
to this point.

\section*{Acknowledgements} This research was not funded by any specific
granting agency in the public, commercial, or not-for-profit sectors.

\bibliographystyle{alpha}
\bibliography{Paper_Beineke_Harary}

\begin{thebibliography}{CMPS14}

\bibitem[AB08]{AB08}
Ron Aharoni and Eli Berger.
\newblock Menger's {T}heorem for {I}nfinite {G}raphs.
\newblock {\em Inventiones mathematicae}, 176(1):1--62, 2008.

\bibitem[BH67]{B01}
Lowell~W.\ Beineke and Frank Harary.
\newblock {T}he {C}onnectivity {F}unction of a {G}raph.
\newblock {\em Mathematika}, 14(2):197–--202, 1967.

\bibitem[BK12]{BK12}
Ralf Bornd{\"o}rfer and Marika Karbstein.
\newblock A {Note} on {M}enger's {T}heorem for {H}ypergraphs.
\newblock Technical Report 12-03, ZIB, Berlin, 2012.

\bibitem[Bod98]{Bodlaender98}
Hans~L.\ Bodlaender.
\newblock A {P}artial k-{A}rboretum of {G}raphs with {B}ounded {T}reewidth.
\newblock {\em Theoretical Computer Science}, 209(1–2):1--45, 1998.

\bibitem[BWO12]{beineke2012topics}
Lowell~W.\ Beineke, Robin~J.\ Wilson, and Ortrud~R.\ Oellermann.
\newblock {\em Topics in Structural Graph Theory}.
\newblock Encyclopedia of Mathematics and its Applications. Cambridge
  University Press, 2012.

\bibitem[CMPS14]{caskurlu}
Bugra Caskurlu, Vahan Mkrtchyan, Ojas Parekh, and K.~Subramani.
\newblock On {P}artial {V}ertex {C}over and {B}udgeted {M}aximum {C}overage
  {P}roblems in {B}ipartite {G}raphs.
\newblock In {\em Theoretical Computer Science}, pages 13--26. Springer, 2014.

\bibitem[Die00]{Diestel00}
Reinhard Diestel.
\newblock {\em Graph Theory}.
\newblock Graduate Texts in Mathematics. Springer, 2000.

\bibitem[EK94]{enomoto1994}
Hikoe Enomoto and Atsushi Kaneko.
\newblock The {C}ondition of {B}eineke and {H}arary on {E}dge-disjoint {P}aths
  some of which are {O}penly {D}isjoint.
\newblock {\em Tokyo Journal of Mathematics}, 17(2):355--357, 1994.

\bibitem[EKM91]{EKM91}
Yoshimi Egawa, Atsushi Kaneko, and Makoto Matsumoto.
\newblock A {M}ixed {V}ersion of {M}enger's {T}heorem.
\newblock {\em Combinatorica}, 11:71--74, 1991.

\bibitem[HT73]{Hopcroft1973}
John Hopcroft and Robert Tarjan.
\newblock Algorithm 447: Efficient algorithms for graph manipulation.
\newblock {\em Communications of the ACM}, 16(6):372--378, 1973.

\bibitem[Mad79]{Mader79}
Wolfgang Mader.
\newblock Connectivity and {E}dge-connectivity in {F}inite {G}raphs.
\newblock In {\em Surveys in Combinatorics (Proceedings of the Seventh British
  Combinatorial Conference), London Mathematical Society Lecture Note Series},
  volume~38, pages 66--95, 1979.

\bibitem[Sch03]{schrijver-book}
Alexander Schrijver.
\newblock {\em Combinatorial Optimization - Polyhedra and Efficiency}.
\newblock Springer, 2003.

\bibitem[SF19]{sadeghi2019}
Elham Sadeghi and Neng Fan.
\newblock On the {S}urvivable {N}etwork {D}esign {P}roblem with {M}ixed
  {C}onnectivity {R}equirements.
\newblock {\em Annals of Operations Research}, 2019.

\bibitem[Wes01]{Wes01}
Douglas~B.\ West.
\newblock {\em Introduction to Graph Theory}, volume~2.
\newblock Prentice-Hall, 2001.

\end{thebibliography}
\label{sec:biblio}\textbf{}

\end{document}